\documentclass{cedram-alco}

\PassOptionsToPackage{hyphens}{url}
\usepackage{hyperref}
\usepackage{amsmath}
\usepackage{amsfonts}
\usepackage{amssymb}
\usepackage{amssymb}
\usepackage{theoremref}
\usepackage{mathrsfs}
\usepackage{enumitem}
\usepackage{amsthm}
\usepackage{faktor}
\usepackage{pstricks-add}
\usepackage{ upgreek }
\usepackage{ dsfont }
\usepackage{pgf,tikz}
\usepackage{mathrsfs}
\usepackage[numbers]{natbib}
\usepackage{notoccite}
\usepackage{mathtools}
\usepackage{subfig}
\usepackage{parskip}
\usepackage{pgfplots}
\usepgfplotslibrary{polar}
\pgfplotsset{compat=newest}
\usepackage{tkz-graph}
\usepackage{pict2e,picture,graphicx}
\usepackage{tikz-cd}

\begingroup
    \makeatletter
    \@for\theoremstyle:=definition,remark,plain\do{
        \expandafter\g@addto@macro\csname th@\theoremstyle\endcsname{
            \addtolength\thm@preskip\parskip
            }
        }
\endgroup

\newcommand{\ZZ}{\mathbb{Z}}
\newcommand{\NN}{\mathbb{N}}
\newcommand{\cG}{\mathcal{G}}
\newcommand{\cA}{\mathcal{A}}
\newcommand{\cB}{\mathcal{B}}
\newcommand{\cS}{\mathcal{S}}
\newcommand{\cC}{\mathcal{C}}
\newcommand{\cL}{\mathcal{L}}
\newcommand{\cV}{\mathcal{V}}
\newcommand{\cE}{\mathcal{E}}
\newcommand{\cH}{\mathcal{H}}
\DeclareMathOperator{\St}{St}
\DeclareMathOperator{\Sch}{Sch}
\DeclareMathOperator{\Cay}{Cay}
\DeclareMathOperator{\lcm}{lcm}
\DeclareMathOperator{\aut}{Aut}
\DeclareMathOperator{\im}{im}
\DeclareMathOperator{\adj}{Adj}
\DeclareMathOperator{\ev}{ev}

\newcommand{\Con}{\mathfrak{B}}

\newcommand{\remark}[1]{{\color{red}}}

\makeatletter
\DeclareRobustCommand{\Arrow}[1][]{
\check@mathfonts
\if\relax\detokenize{#1}\relax
\settowidth{\dimen@}{$\m@th\rightarrow$}
\else
\setlength{\dimen@}{#1}
\fi
\sbox\z@{\usefont{U}{lasy}{m}{n}\symbol{41}}
\begin{picture}(\dimen@,\ht\z@)
\roundcap
\put(\dimexpr\dimen@-.7\wd\z@,0){\usebox\z@}
\put(0,\fontdimen22\textfont2){\line(1,0){\dimen@}}
\end{picture}
}
\makeatother
\newcommand{\veryshortrightarrow}{\hspace{.05mm}\scalebox{1.5}{\Arrow[.3cm]}\hspace{.05mm}}

\newcommand{\rto}{{\color{red}\veryshortrightarrow}}

\newcommand{\cto}{{\color{cyan}{\veryshortrightarrow}}}

\usepackage{xifthen}
\usetikzlibrary{positioning}

\newcommand{\tikzproductgraph}[3]{
	\begin{tikzpicture}[scale=1.5]
        
            \tikzstyle{every node}=[font=\small];
        	\pgfmathsetmacro{\p}{#1};            
			\pgfmathsetmacro{\q}{#2} ;           
			\pgfmathsetmacro{\pm}{\p-1};            
			\pgfmathsetmacro{\pmm}{\p-2};            
			\pgfmathsetmacro{\pmmm}{\p-3};            
			\pgfmathsetmacro{\qm}{\q-1} ;           
			\pgfmathsetmacro{\qmm}{\q-2}  ;  
			\pgfmathsetmacro{\dnne}{0}  ;  
			
			\pgfmathsetmacro{\numarms}{#3}        
 
            \foreach \j in {0,...,\pm}
            {
                \foreach \i in {0,...,\qm}
                {
                    \ifnum \j=0
                    \ifnum \i=0
                    \ifnum \numarms=1
                        \draw[fill] (\i,\j) circle [radius=0pt];
                        \node[above] at (\i,\j) {${v_-}$};
			            \pgfmathsetmacro{\dnne}{1}  ;  
                    \fi
                    \ifnum \numarms=2
                        \draw[fill] (\i,\j) circle [radius=0pt];
                        \node[above] at (\i,\j) {${v_+}$};
			            \pgfmathsetmacro{\dnne}{1}  ;  
                    \fi
                    \fi
                    \fi
                    
                    \ifnum \j=\pm
                    \ifnum \i=0
                    \ifnum \numarms>1
                        \draw[fill] (\i,\j) circle [radius=0pt];
                        \node[above] at (\i,\j) {${u_+}$};
			            \pgfmathsetmacro{\dnne}{1}  ;  
                    \fi
                    \fi
                    \fi
                   
                    \ifnum \dnne=0
                    \pgfmathsetmacro{\ip}{\i+1}
                        \draw[fill] (\i,\j) circle [radius=0pt];
                        \node[above] at (\i,\j) {${(\i,\j)}$};
                    \fi
                }
            }
            
            \ifnum \numarms=1
                \draw[fill] (0,-1) circle [radius=0pt];
                \node[below] at (0,-1) {$v_+$};            
           	\fi
            \ifnum \numarms=2
                \draw[fill] (0,-1) circle [radius=0pt];
                \node[below] at (0,-1) {$v_-$};   
                \draw[fill] (0,\p) circle [radius=0pt];
                \node[above] at (0,\p) {$u_-$};         
            \fi

            \foreach \j in {0,...,\pm}
            {
                \foreach \i in {0,...,\qmm}
                {
                    \pgfmathsetmacro{\ip}{\i+1}
				    \draw[->, red] (\i,\j) -- (\ip,\j);
                }
                \draw[fill] (\qm,\j) circle [radius=0pt];
                \draw[->, red, out=150, in=30] (\qm,\j) to (0,\j);
            }

            \ifnum \numarms=0
	            \ifnum \p=2
	                \draw[->, cyan, out=160,in=20] (1,\pm) to (0,1);	        
	            \else
	            	\foreach \j in {1,...,\pmm}
	                {
	                    \pgfmathsetmacro{\jp}{\j+1};
	                    \draw[->, cyan] (0,\j) to (0,\jp);
	                }
	                \draw[->, cyan] (1,\pm) to (0,1);	          
	            \fi
	            \draw[->, cyan, out=-20, in=-160] (0,\pm) to (1,\pm);
	            
	            \ifnum \p=2
	                \draw[->, cyan, out=160,in=20] (1,\pmm) to (0,0);	        
	            \else
	            	\foreach \j in {0,...,\pmmm}
	                {
	                    \pgfmathsetmacro{\jp}{\j+1};
	                    \draw[->, cyan] (1,\j) to (1,\jp);
	                }
	                \draw[->, cyan] (1,\pmm) to (0,0);
	            \fi
	            \draw[->, cyan, out=-20, in=-160] (0,0) to (1,0);
	        \fi   
	        
            \ifnum \numarms=1
	            \ifnum \p=2
	                \draw[->, cyan, out=160,in=20] (1,\pm) to (0,1);	        
	            \else
	            	\foreach \j in {1,...,\pmm}
	                {
	                    \pgfmathsetmacro{\jp}{\j+1};
	                    \draw[->, cyan] (0,\j) to (0,\jp);
	                }
	                \draw[->, cyan] (1,\pm) to (0,1);	          
	            \fi
	            \draw[->, cyan, out=-20, in=-160] (0,\pm) to (1,\pm);
	            
	            \ifnum \p=2
	                \draw[->, cyan, out=65,in=-155] (0,-1) to (1,0);
	            \else
	            	\foreach \j in {0,...,\pmmm}
	                {
	                    \pgfmathsetmacro{\jp}{\j+1};
	                    \draw[->, cyan] (1,\j) to (1,\jp);
	                }
	                \draw[->, cyan] (0,-1) to (1,0);
	            \fi
	            \draw[->, cyan] (1,\pmm) to (0,-1);
	        \fi        
            
            \ifnum \numarms=2
	            \ifnum \p=2
	                \draw[->, cyan, out=155,in=-65] (1,\pm) to (0,\p);	          
	            \else
                    \ifnum \p=3
                    \else
	            	    \foreach \j in {1,...,\pmmm}
                        {
                            \pgfmathsetmacro{\jp}{\j+1};
                            \draw[->, cyan] (0,\j) to (0,\jp);
                        }
	                    \draw[->, cyan] (1,\pm) to (0,1);
                    \fi
                    \draw[->, cyan, out=110, in=-110] (0,\pmm) to (0,\p);
                
	            \fi
	            \draw[->, cyan] (0,\p) to (1,\pm);
	            
	            \ifnum \p=2
	                \draw[->, cyan, out=65,in=-155] (0,-1) to (1,0);
	            \else
	            	\foreach \j in {0,...,\pmmm}
	                {
	                    \pgfmathsetmacro{\jp}{\j+1};
	                    \draw[->, cyan] (1,\j) to (1,\jp);
	                }
	                \draw[->, cyan] (0,-1) to (1,0);
	            \fi
	            \draw[->, cyan] (1,\pmm) to (0,-1);
	        \fi 

            \foreach \i in {2,...,\qm}
            {
                \foreach \j in {0,...,\pmm}
                {
                    \pgfmathsetmacro{\jp}{\j+1};
                    \draw[->, cyan] (\i,\j) to (\i,\jp);
                }
                \draw[->, cyan, out=-70, in=70] (\i,\pm) to (\i,0);
            }

	\end{tikzpicture}
}

\newcommand{\tikztttzarms}{
	\begin{tikzpicture}
                      
            \node (1) {$1$};
            \node[right= of 1] (0) {$0$};
        \node[below= of 1] (3) {$3$};
        \node[below= of 0] (2) {$2$};
        \node[below= of 3] (10) {$10$};
        \node[below= of 2] (5) {$5$};
        \node[below= of 10] (9) {$9$};
        \node[left= of 10] (11) {$11$};
        \node[below= of 11] (8) {$8$};
        \node[below= of 5] (7) {$7$};
        \node[right= of 5] (4) {$4$};
        \node[right= of 7] (6) {$6$};

            \draw[red] (1) -- (3);    
            \draw[red] (0) -- (2);    
            \draw[red] (10) -- (9);    
            \draw[red] (11) -- (8);    
            \draw[red] (7) -- (5);    
            \draw[red] (6) -- (4);
            \draw[blue] (10) -- (11); 
            \draw[blue] (8) -- (9);
             \draw[blue]   (6) -- (7);
             \draw[blue]   (4) -- (5);
             \draw[blue]   (0) -- (1);
             \draw[blue]   (2) -- (3);

             \draw[green, out=30, in=-210] (1) to (0);
             \draw[green] (9) -- (11);
             \draw[green] (2) -- (5);
             \draw[green] (3) -- (10);
             \draw[green, out=-30, in=210] (8) to (7);
             \draw[green, out=60, in=-60] (6) to (4);

	\end{tikzpicture}
}
\newcommand{\tikzttttarms}{
	\begin{tikzpicture}
                      
            \node[circle, draw, thick] (1) {$1$};
            \node[right = of 1] (0) {$0$};
        \node[below= of 1] (3) {$3$};
        \node[below= of 0] (2) {$2$};
        \node[below= of 3] (10) {$10$};
        \node[below= of 2] (5) {$5$};
        \node[below= of 10] (9) {$9$};
        \node[left= of 10] (11) {$11$};
        \node[below= of 11] (8) {$8$};
        \node[below= of 5] (7) {$v_-$};
        \node[below= of 8] (7b) {$v_+$};
        \node[right= of 5] (4) {$4$};
        \node[right= of 7] (6) {$6$};

            \draw[red] (1) -- (3);    
            \draw[red] (0) -- (2);    
            \draw[red] (10) -- (9);    
            \draw[red] (11) -- (8);    
            \draw[red] (7) -- (5);    
            \draw[red] (6) -- (4);
            \draw[blue] (10) -- (11); 
            \draw[blue] (8) -- (9);
             \draw[blue]   (6) -- (7);
             \draw[blue]   (4) -- (5);
             \draw[blue]   (0) -- (1);
             \draw[blue]   (2) -- (3);

             \draw[green, out=30, in=-210] (1) to (0);
             \draw[green] (9) -- (11);
             \draw[green] (2) -- (5);
             \draw[green] (3) -- (10);
             \draw[green] (7b) to (8);
             \draw[green, out=60, in=-60] (6) to (4);
	
	\end{tikzpicture}

}
\newcommand{\tikztttfarms}{
	\begin{tikzpicture}
                      
            \node (1) {$1$};
            \node[right= of 1] (0) {$u_+$};
            \node[above=of 1] (0b) {$u_-$};
        \node[below= of 1] (3) {$3$};
        \node[below= of 0] (2) {$2$};
        \node[below= of 3] (10) {$10$};
        \node[below= of 2] (5) {$5$};
        \node[below= of 10] (9) {$9$};
        \node[left= of 10] (11) {$11$};
        \node[below= of 11] (8) {$8$};
        \node[below= of 5] (7) {$v_+$};
        \node[below= of 8] (7b) {$v_-$};
        \node[right= of 5] (4) {$4$};
        \node[right= of 7] (6) {$6$};

            \draw[red] (1) -- (3);    
            \draw[red] (0) -- (2);    
            \draw[red] (10) -- (9);    
            \draw[red] (11) -- (8);    
            \draw[red] (7) -- (5);    
            \draw[red] (6) -- (4);
            \draw[blue] (10) -- (11); 
            \draw[blue] (8) -- (9);;
             \draw[blue]   (6) -- (7);
             \draw[blue]   (4) -- (5);
             \draw[blue]   (0) -- (1);
             \draw[blue]   (2) -- (3);
 
             \draw[green] (1) to (0b);
             \draw[green] (9) -- (11);
             \draw[green] (2) -- (5);
             \draw[green] (3) -- (10);
             \draw[green] (7b) to (8);
             \draw[green, out=60, in=-60] (6) to (4);

	\end{tikzpicture}
}

\title[Telescopic groups and symmetries of combinatorial maps]{Telescopic groups and symmetries of combinatorial maps}

\author[\initial{R.} Bottinelli]{\firstname{R\'emi} \lastname{Bottinelli}}

\address{Institut de Math\'ematiques, Universit\'e de Neuch\^atel, Rue Emile-Argand 11, CH-2000 Neuch\^atel, Suisse / Switzerland}

\email{remi.bottinelli@unine.ch}

\author[\initial{L.} Grave de Peralta]{\firstname{Laura} \lastname{Grave de Peralta}}

\address{Institut de Math\'ematiques, Universit\'e de Neuch\^atel, Rue Emile-Argand 11, CH-2000 Neuch\^atel, Suisse / Switzerland}

\email{laura.grave@unine.ch}

\author[\initial{A.} Kolpakov]{\firstname{\\Alexander} \lastname{Kolpakov}}

\address{Institut de Math\'ematiques, Universit\'e de Neuch\^atel, Rue Emile-Argand 11, CH-2000 Neuch\^atel, Suisse / Switzerland}

\email{kolpakov.alexander@gmail.com}

\thanks{The authors were supported by the Swiss National Science Foundation project no.~PP00P2-144681/1}

 

\begin{document}

\begin{abstract}
    In the present paper, we show that many combinatorial and topological objects, such as maps, hypermaps, three-dimensional pavings, constellations and branched coverings of the two--sphere admit any given finite automorphism group. This enhances the already known results by Frucht, Cori -- Mach\`i,  \v{S}ir\'{a}\v{n}  -- \v{S}koviera, and other authors. We also provide a more universal technique for showing that ``any finite automorphism group is possible'', that is applicable to wider classes or, in contrast, to more particular sub-classes of said combinatorial and geometric objects. Finally, we show that any given finite automorphism group can be realised by sufficiently many non-isomorphic such entities (super-exponentially many with respect to a certain combinatorial complexity measure).
\end{abstract}

\maketitle

\section{Introduction}\label{intro}

A combinatorial (oriented, labelled) map is a triple $M = (D; R, L)$ where $D$ is a non-empty finite set (called the set of darts) and $R$ and $L$ are two permutations of $D$ with $L^2 = \mathrm{id}$. The orbits of $L$ are conventionally called the edges of $M$, the orbits of $R$ are its vertices, and the orbits of $R^{-1}L$ are its faces. The map $M$ is called connected if the group $\langle R, L \rangle$ acts transitively on $D$. Unless otherwise stated, we shall assume all maps to be connected.

A topological (oriented) map $M = (\Sigma_g; \Gamma)$ is an oriented (connected) genus $g\geq 0$ surface with an embedded graph $\Gamma$ such that the complement $\Sigma_g \setminus \Gamma$ is a collection of disjoint topological discs. By providing a labelling on the half-edges of $\Gamma$ (thus defining its labelled darts), and thus obtaining a labelled topological map, one can recover the permutations $L$ and $R$, so that $L$ encodes the identification of half-edges into edges, and $R$ encodes the positive cyclic order of half-edges around each vertex. Vice versa, provided a combinatorial map, one can recover its corresponding topological labelled counterpart by creating the faces (which are discs) by following the cycles of $R^{-1}L$, and then identifying their boundaries by using $L$. For more details, cf. \cite{CM-Survey, JoSi, Tutte}.

One can define a more elaborate class of combinatorial objects (and the corresponding topological objects) such as hypermaps \cite{Cori}. A triple $H = (D; R, L)$, where $D$ is a non-empty  finite set of darts and $R$, $L$ are permutation of $D$, is called an (oriented, labelled) hypermap. The orbits of $L$ are called the hyper-edges of $H$, the orbits of $R$ are its hyper-faces. A hypermap  $H$ is connected whenever the group $\langle R, L \rangle$ acts transitively on $D$ (which will be our standing assumption).

A hypermap also naturally appears in the setting of an orientable genus $g$ surface $\Sigma_g$ and a graph $\Gamma$ embedded in $\Sigma_g$ that satisfies the following properties:
\begin{itemize}
\item[1)] the complement $\Sigma_g \setminus \Gamma$ is a union of topological discs called faces,
\item[2)] the faces are properly two-colourable (e.g. into black and white), i.e. faces of the same colour intersect only at vertices of $\Gamma$, and
\item[3)] the corners of the white faces are labelled with the numbers $1, 2, 3, \dots$ in some fashion, and a black face corner label is equal to the adjacent white face corner label, when moving clockwise around their common vertex.
\end{itemize}
Then  $H = (\Sigma_g; \Gamma)$ is an oriented labelled topological hypermap. 

The correspondence between the topological and combinatorial definitions is as follows: 
\begin{itemize}
\item[1)] each disjoint cycle of $R$ is obtained from recording the corner labels of a white face in a counter-clockwise direction,
\item[2)] each disjoint cycle of $L$ is obtained from recording the corner labels of a black face in a counter-clockwise direction,
\item[3)] each disjoint cycle of $R^{-1}L$ is obtained from recording the labels around a vertex in a counter-clockwise direction,
\end{itemize}
We remark that condition (3) above is a consequence of (1) and (2).

The set of face labels becomes the set of darts of $H$, the white faces become hyper-faces of $H$ and the black faces become hyper-edges of $H$. Thus, the combinatorial and topological definitions of an oriented labelled hypermap agree. 

If $L^2 = \mathrm{id}$, then each bigon in the hypermap 
$H = ( D; R, L )$ can be interpreted as a pair of darts pointing in opposite directions, and thus $H$ becomes a map, as defined above. 

We say that two oriented labelled (hyper-)maps $M_1 = (D; R_1, L_1)$ and $M_2 = (D; R_2, L_2)$ are isomorphic if, in the combinatorial setting, there exists a permutation $T$ of $D$ such that  $T R_1 = R_2 T$ and $T L_1 = L_2 T$. In the topological setting, two oriented labelled (hyper-)maps $M_1 = (\Sigma_g; \Gamma_1)$ and $M_2 = (\Sigma_g; \Gamma_2)$ are isomorphic if there exists an orientation-preserving homeomorphism $\tau: \Sigma_g \rightarrow \Sigma_g$ such that $\tau(\Gamma_1) = \Gamma_2$ and the labelling of the corresponding half-edges is respected. 

A rooted isomorphism will require only the root (a dedicated labelled dart) of one (hyper-)map to be carried to the root of another. 

Finally, an isomorphism is not required to respect the dart labelling, nor the roots. 

The above definition allows us to generalise the setting of maps to higher-dimensional objects, the so-called pavings. Namely, as defined in \cite{AK}, a three-dimensional oriented combinatorial map or, simply, a (combinatorial) paving, is a quadruple $P = ( D; R, L, V )$, where $D$ is a non-empty set of darts and $R, L, V$ are permutations of $D$ such that $H_P = ( D; R, L )$ is a map (not necessarily connected), and

\begin{itemize}
\item[1)] the product $L V$ is an involution,
\item[2)] the product $V R^{-1}$ is an involution,
\item[3)] none of the above involutions have fixed points.
\end{itemize}

A paving $P$ is connected if the group $\langle L, R, V \rangle$ acts transitively on $D$. The notion of (labelled, rooted) isomorphism for oriented combinatorial pavings is analogous to the one for combinatorial maps. 

We may also think of $P$ as a quadruple $P = ( D; L, S, T )$,  where $D$ is the set of darts and $L, S, T$ are its involutions without fixed points. In this case it is easy to see that letting $V = L S$ and $R = T L S $ produces the initial definition.  As in the case of two-dimensional maps, a combinatorial paving $P$ has a topological realisation which, however, is not always a three-dimensional manifold (however, it's always a pseudo-manifold).

In order to assemble an oriented cellular complex $M_P$, as described in \cite{Spehner}, we first produce its underlying map $H_P = ( D; R, L )$, and realise each connected component of $H$ as a topological map, i.e. as a surface $\Sigma^i$ with an embedded graph $\Gamma^i$, $i=1, 2, \dots, m$, having labelled half-edges. Each surface $\Sigma^i$ represents the boundary of a handle-body $B^i$, and then the handle-bodies $B^i$ become identified along their boundaries in order to produce a labelled oriented cellular complex representing $P$ topologically. Indeed, the faces of $\Sigma^i$'s defined by the permutation $R^{-1} L$ are identified in accordance with the permutation $V$, and the conditions (1), (2), and (3) above ensure that one face cannot be identified to multiple disjoint counterparts  (implied by (1) and (2)), and edges or faces cannot bend onto themselves (implied by (3)). Also, conditions (1) and (2) ensure that $M_P$ is an orientable topological space.

There are other generalisations of maps, hypermaps and pavings, such as constellations, cf. the monograph \cite{LZ} for more information and references. 

One of the basic questions is understanding possible symmetries, or automorphisms (i.e. unrooted self-isomorphisms), of any of the above defined objects. Those can be understood by means of building a one-to-one correspondence between a class of rooted (hyper-)maps $\mathcal{M}$ (or isomorphism classes of maps) on a set of dart $D$ and (usually, torsion-free) subgroups (or their conjugacy classes) of a given single group $\Delta^+$. This correspondence will associate to each map $M \in \mathcal{M}$ a subgroup $H_M \subset \Delta^+$ of index $|D|$. The origins of this technique draw back to the paper by Jones and Singerman \cite{JoSi}, and have been developed more in the recent works by Breda, Mednykh and Nedela \cite{BrMeNe}, Mednykh and Nedela \cite{MN1, MN2} for the purpose of solving Tutte's problem of (hyper-)map classification, cf. also \cite{CK1, CK2}.

Let us consider the case of maps, as described in \cite{BrMeNe, JoSi}. Namely, the rooted maps on $n$ darts (where the root is always supposed to be marked $1$) are in a one-to-one correspondence with index $n$ free subgroups of $\Delta^+ = \mathbb{Z}*\mathbb{Z}_2$. Indeed, each free subgroup $H < \Delta^+$ of index $n$ produces a set of cosets $D = \Delta^+ / H$ of cardinality $n$, which can be considered as a set of darts. The root dart here is the identity coset. A subgroup of $\Delta^+$ is torsion-free if and only if it is free, as a consequence of Kurosh's theorem. Thus $\Delta^+ = \mathbb{Z}*\mathbb{Z}_2 \cong \langle \sigma \rangle * \langle \alpha \rangle$ acts on $D$ transitively, and its generators $\sigma$ and $\alpha$ give rise to permutations $R$ and $L$ acting transitively on $D$. Thus, we obtain a map $M_H$ corresponding to a free subgroup $H < \Delta^+$. Vice versa, given a map $( D; R, L )$, we have a homomorphism $S: \Delta^+ \rightarrow \langle R, L \rangle$ by setting $S(\sigma) = R$, $S(\alpha) = L$. The homomorphism $S$ defines an action of $\Delta^+$ on $D$, and the subgroup corresponding to $M$ is $H_M = Stab(1) < \Delta^+$. 

The above correspondence between the free subgroups of $\Delta^+$ and rooted maps can be extended to the case of hypermaps (with $\Delta = \mathbb{Z}*\mathbb{Z} \cong F_2$), or the so-called $(p,q)$-hypermaps (with $\Delta^+ = \mathbb{Z}_p*\mathbb{Z}_q$ \cite{CK1}), or $3$-dimensional maps (also called pavings \cite{AK, Spehner}, with $\Delta^+ = \mathbb{Z}_2*\mathbb{Z}_2*\mathbb{Z}_2$ \cite{CK2}). 

The isomorphisms classes of all aforementioned objects correspond to  the conjugacy classes of free subgroups of $\Delta^+$ \cite[Theorem~3.7]{JoSi}. The symmetries (i.e. unrooted self-isomorphisms) of a (hyper-)map $M$ corresponding to a subgroup $H < \Delta^+$ form a group isomorphic to $N(H)/H$, where $N(H) = \{ g \in \Delta^+ \, | \, g H g^{-1} = H \}$ is the normaliser of $H$ in $\Delta^+$ \cite[Theorem~3.8]{JoSi}.

In the sequel we shall study a more abstract question, namely the property of free products of cyclic groups being ``telescopic'', cf. Definition~\ref{defn:telescopic}. Such a free product $T$ being telescopic allows us to realise any finite group $\Gamma$ as the ``symmetry group'' $N(H)/H$ of a suitable finite-index subgroup $H \leq T$. Thus, one of our main results is the following statement, cf. Theorem~\ref{thm:free-products-are-telescopic}.

\begin{thm*}
Any free product of at least two non-trivial cyclic groups is freely telescopic, except for the infinite dihedral group $D_\infty \cong \ZZ_2*\ZZ_2$.
\end{thm*}

If we allow the index of $H$ to be sufficiently large, depending on the cardinality of $\Gamma$, then a great deal of same index subgroups $H$ with $N(H)/H \cong \Gamma$ can be obtained, cf. Theorem~\ref{thm:asympt}. More precisely, the following holds. 

\begin{thm*}
Let $T$ be a finite free product of cyclic groups, different from $\ZZ_2*\ZZ_2$. Then for any finite group ${\Gamma}$, there exist constants $A > 1$, $B > 0$ and $M\in\NN$ such that for all $d{\geq}M$ the set $F(T, {\Gamma}, d) = \{\text{free subgroups } H \leq T \text{ of index } \leq d \text{ with } N_T(H)/H{\cong}{\Gamma}, \text{ up to conjugacy}\}$ has cardinality ${\geq}A^{B d\log d}$.
\end{thm*}

Finally, we translate our group-theoretic statements into the combinatorial language of (hyper-)maps and pavings, cf. Theorems~\ref{thm:p-q-hypermaps-symmetries}--\ref{thm:constellations-symmetries}. Such a transition from combinatorics to groups, to combinatorics again is an integral part of our approach. First, we want to obtain some information about symmetries of a sufficiently complicated combinatorial object. Next, we translate our questions about symmetries into a question about the existence of (torsion-free) subgroups of a free product of cyclic groups with some condition on their normalisers. This condition is formulated in terms of combinatorial automorphisms of the subgroup's Schreier graph, by analogy to the approach introduced in \cite{KM, We}. The symmetries of the corresponding Schreier graphs appear more amenable to combinatorial analysis, which finally provides us with the desired results both in group-theoretic and combinatorial terms. 

\vspace*{-0.25in}

\begin{rema*}
Soon after a draft of this paper appeared on the arXiv, the authors were notified by Gareth A. Jones that his paper \cite{Jones-arXiv} contains similar results for a wider class of group. In particular, by \cite[Theorem 3]{Jones-arXiv}, all hyperbolic (extended) triangle groups are shown to be ``finitely abundant'', which is equivalent to being telescopic for the non-compact ones among them. Also, all subgroups produced in \cite[Theorem 3]{Jones-arXiv} are, in fact, torsion-free. The methods used in \cite{Jones-arXiv} and in our paper differ substantially, as well as the emphasis in our work is on the quantitative aspects, such as counting of combinatorial objects with given symmetries. 
\end{rema*}

\section{Preliminaries}

We first establish the necessary notation and provide some basic definitions. Let $G$ be a group, and $H$ be a subgroup of $G$. Let $N_G(H) = \{ g \in G\ :\ gHg^{-1} = H \}$ denote the normaliser of $H$ in $G$.  

\begin{defi}\label{defn:telescopic}
We say that a group $T$ is \textit{telescopic} if for every finite group $\Gamma$ there exists a finite-index subgroup $H \leq T$ such that $N_T(H)/H \cong \Gamma$. 
\end{defi}

\begin{defi}
If in the above definition we can always choose $H$ to be a free subgroup of $T$, we say that $T$ is \textit{freely telescopic}. 
\end{defi}

\begin{defi}
A (di-) graph is a tuple $(V,E,{\iota}:E{\to}V,{\tau}:E{\to}V)$, where $V$ is the set of \emph{vertices}, $E$ is the set of directed \emph{edges} and ${\iota}$, resp. ${\tau}$, assigns to each edge $e$ its initial vertex (or origin) ${\iota}(e)$, resp. its terminal vertex (or terminus) ${\tau}(e)$. We shall write $\St_+v := \{e{\in}E\ :\ {\iota}(e) = v\}$ and $\St_-v := \{e{\in}E\ :\ {\tau}(e) = v\}$. A morphism of graphs ${\phi}:(V_{1},E_{1},{\iota}_{1},{\tau}_{1}) {\to} (V_{2},E_{2},{\iota}_{2},{\tau}_{2})$ consists of a pair of maps ${\phi}_V:V_{1}{\to}V_{2}$ and ${\phi}_E:E_{1}{\to}E_{2}$ such that ${\iota}_{2}{\phi}_E = {\phi}_V{\iota}_{1}$ and ${\tau}_{2}{\phi}_E = {\phi}_V{\tau}_{2}$. If $\cG$ is a (di-)graph, then $V\cG$ will denote its set of vertices, and $E\cG$ will be its set of edges.

A graph is \textit{labelled} by the elements of a set $S$ if a map ${\mu}:E{\to}S$ is given. Let us write $S$-digraph for an $S$-labelled digraph. A labelled $S$-digraph $\cG$ is called \emph{folded} (cf.~\cite{St}) if, for any $v \in V\cG$, the restrictions of ${\mu}$ to $\St_+v$ and $\St_-v$ are injective, and $\cG$ is called \emph{regular} if they are bijective. A morphism of $S$-digraphs is a morphism of digraphs satisfying ${\mu}_{2}{\phi}_E = {\mu}_{1}$.
\end{defi}

Recall that if $G$ is a group generated by a set $S$, and $H{\leq}G$ is a subgroup of $G$, then the Schreier graph $\Sch_{G,S}(H)$ of $H$ is an $S$-labelled regular graph having vertex set the right cosets of $H$, and an edge $H\cdot g \xrightarrow{s} H\cdot gs$ from $H\cdot g$ to $H\cdot gs$, labelled $s$, for each element $s{\in}S$ and each coset $H\cdot g$. The Cayley graph $\Cay(G, S)$ of $G$ is defined as the Schreier graph of the trivial subgroup $\{ \mathrm{id} \}$ of $G$. Both Cayley and Schreier graphs are generally considered with a basepoint: the coset $H\cdot e = H$. When graphs with basepoints are considered, their morphisms are assumed to send basepoints to basepoints.

If a graph $\mathcal{G}$ has label set $S \subset G$, for some group $G$, and a vertex $v \in V\cG$ is specified, then the set
\[
    L(\mathcal{G},v) := \{\text{labels of loops at $v$, evaluated in $G$}\}
\]
forms a group, called the language of $\mathcal{G}$ at $v$ (an empty loop gives the identity of $G$, paths concatenation corresponds to taking products, and reversing paths corresponds to taking inverses). Let $\mathrm{ev}: \{\text{words in }S \sqcup S^{-1}\} {\to} G$ denote the evaluation map, so that $L(\mathcal{G},v) = \mathrm{ev}{\circ}{\mu}\left(\text{``loops at $v$''}\right)$. 

We also present a few key results that we make use of, most of which can be found in \cite{KM, St, We}. First, recall Kurosh's theorem.

\begin{thm}[Kurosh's Subgroup Theorem]
    Let $G_{1}, \dots, G_n$ be groups, and $G := *_{i=1}^n G_i$ be their free product.
    Then, any subgroup $H {\leq} G$ of $G$ has the form:
    \[
        (*_{i=1}^n*_{j=1}^{m_i} w_{ij}H_{ij}w_{ij}^{-1} )* F(X)
    \]
    where each $H_{ij}$ is a subgroup of $G_i$, $F(X)$ is a free subgroup generated by a subset $X$ of $G$, and $w_{ij}$ is an element of $G$, for $1{\leq}i{\leq}n$ and $1{\leq}j{\leq}m_i$.
\end{thm}
\begin{proof}
    See, for instance, the monograph~\cite[pp. 56 -- 57]{Se}.
\end{proof}

\begin{lemm}[{\cite[Lemma 7.5]{KM}}]\label{lemma:basepoint-conjugate}
If $\mathcal{G}$ is an $S$-digraph with $S$ a subset of a group $G$, then for any vertices $v_{1},v_{2}$ connected by a path $p:v_{1} \rightsquigarrow v_{2}$ with $g:=\ev \circ {\mu}(p)$:
    \[
        L(\mathcal{G},v_{1}) = gL(\mathcal{G},v_{2})g^{-1}.
    \]
\end{lemm}
\begin{proof}
If $l$ is a loop at $v_{2}$, then $plp^{-1}$ is a loop at $v_{1}$, and $\ev{\mu}(plp^{-1}) = g\ev{\mu}(l)g^{-1}$, so that the conjugation by $g$ maps elements of $L(\cG,v_{2})$ to elements of $L(\cG,v_{1})$, and thus $gL(\cG,v_{2})g^{-1} {\subseteq} L(\cG,v_{1})$. Symmetrically, $g^{-1}L(\cG,v_{1})g {\subseteq} L(\cG,v_{2})$, and the result follows.
\end{proof}

\begin{lemm}[{\cite[Lemma 4.2]{KM}}]
If $\cA,\cB$ are $S$-digraphs and $\cB$ is folded, then for any vertices $v{\in} V\cA$ and $u{\in} V\cB$, there exists at most one morphism of $S$-digraphs ${\phi}:\cA{\to}\cB$ satisfying ${\phi}(v)=u$. 
\end{lemm}
\begin{proof}
Follow paths; cf.~\cite[Lemma 4.2]{KM} or \cite[5.1 (c)]{St}.
\end{proof}

Let $\ZZ_{p_i} = \langle s_i\, |\, s_i^{p_i} \rangle$ denote the cyclic group of order $p_i \in \mathbb{N} \cup \{\infty\}$, for $p_i\geq 2$, while setting $p_i = \infty$ yields $\ZZ$ (which will be our standard notation for the rest of the paper).

\begin{lemm}[{\cite[Theorem 1.2]{Sto}}]\label{lemma:bijection-subgroup-graph}
Let $G := \ZZ_{p_1}* \dots *\ZZ_{p_n}$ be a free product of cyclic groups with generators $S = \{s_1, \dots, s_n\}$, assuming $\ZZ_{p_i} = \langle s_i \rangle$.
    There is a bijection between the sets
    \begin{enumerate}
        \item[$A:=$] ``subgroups of $G$'' and
        \item[$B:=$] ``connected, regular $S$-digraphs with a basepoint, such that for any $s_i$, the edges labelled by $s_i$ form cycles of length dividing $p_i < \infty$, up to isomorphism''.
    \end{enumerate}
    Moreover, free (equivalently, torsion-free) subgroups correspond to graphs with $s_i$-labelled cycles of length exactly $p_i < \infty$, and under this equivalence, the index of a subgroup equals the number of vertices in the corresponding graph.
\end{lemm}
If, in a labelled digraph as in the lemma above, a cycle labelled by $s_i$ has length a proper divisor of $p_i < \infty$, such a cycle will be called degenerate, following~\cite{Sto}. An element of $B$ is a $(G,S)$-Schreier graph, which is called \textit{non-degenerate}, if it contains no degenerate cycle.

\begin{proof}[Sketch of proof.]
The Schreier graph of the quotient $\Sch_{G,S}({\cdot})$ provides one direction of the equivalence, while the language at the root $L({\cdot})$ proves the other. In order to show the equivalence ``free subgroup'' ${\Leftrightarrow}$ ``no degenerate cycle'', one uses the fact that a torsion element in $G$ must be conjugate to an element of one of the factors, and vice versa.
\end{proof}

The equivalence of Lemma~\ref{lemma:bijection-subgroup-graph} can actually be generalised to arbitrary finitely generated groups if one does not care about freeness: this is done (using a slightly different language) in~\cite[Theorem 3.5]{We}.

\begin{prop}\label{lemma:N-over-H-as-auto}
Let $\cS(H) := \Sch_{G,S}(H)$. Then $\aut(\cS(H)) \cong N_G(H)\diagup H$.
\end{prop}

In the above statement, $\aut$ denotes the group of automorphisms of a labelled digraph without basepoint.

\begin{proof}
Let $N := N_G(H)$, and let us consider the following map:
\begin{alignat*}{10}
\Phi : N &\longrightarrow \aut(\cS(H))\\
n &\longmapsto \left( \phi_n : Hg \mapsto Hng \right)
\end{alignat*}
Then, $\Phi$ is a well-defined surjective group homomorphism with kernel exactly $H$. 

It is routine to check that $\Phi$ is well-defined and is a homomorphism: this fact depends on $N$ being the normaliser of $H$. Let us verify its surjectivity. 

Let $\phi$ be an element of $\aut(\cS(H))$ and let $Hg$ be the image of $H$ under $\phi$.
Since ${\phi}$ is an automorphism, we have that
\[
    H = L(\cS(H),He) = L({\phi}(\cS(H)),{\phi}(He)) = L(\cS(H),Hg),
\]
Also, we know that $L(\cS(H),Hg) = g^{-1}L(\cS(H),He)g$,  since changing the basepoint changes the language by conjugation, as in Lemma~\ref{lemma:basepoint-conjugate}. This implies $H = g^{-1}Hg$, and thus $g{\in}N$. Since $\cS(H)$ is a regular graph, there is a unique morphism sending $H$ to $Hg$, so that $\phi=\phi_g$. Therefore, $\Phi$ is surjective.

Finally, let us verify that $\ker {\Phi} = H$. Let $n$ be an element of $\ker \Phi$.  Then $\phi_n = \mathrm{id}_{\cS(H)}$, which implies that $Hn = H$, and $n{\in}H$.
Conversely, for any $h {\in} H$, ${\phi}_h$ is readily seen to be the identity map, since $Hh = H$. Therefore, $\ker {\Phi} = H$ and the claim follows by the first isomorphism theorem.
\end{proof}

\section{Free products of cyclic groups}

In this section we show that any free product of (non-trivial) cyclic groups with at least two factors is freely telescopic, with the obvious exception of $D_\infty \cong \ZZ_2*\ZZ_2$, the infinite dihedral group.

From now on, let ${\Gamma}$ denote a finite group with generating set $S$. Let $T$ be a finite free product of cyclic groups $\ZZ_{p_1}*\dots*\ZZ_{p_n}$, and $X$ be the natural choice of its generators (one per factor). We always assume that $p_i \geq 2$ and, if $T=\ZZ_p*\ZZ_q$, also that $p \geq q \geq 2$, while $p \geq 3$. 

We will proceed as follows, in order to prove that any $T$ as above is freely telescopic, or equivalently, that any finite group ${\Gamma}$ is isomorphic to a quotient $N_T(H)/H$ for $H$ a free finite-index subgroup of $T$.

\subsection{Plan of proof}\label{proof-plan}
\begin{enumerate}
\item Some algebraic arguments (Lemmas \ref{lemma:telescopic-lcm} -- \ref{lemma:telescopic-Z}) using Kurosh's Subgroup Theorem reduce the problem to free products of the form $T=\ZZ_p*\ZZ_q$ ($p \geq 3, q \geq 2$) and $T=\ZZ_2*\ZZ_2*\ZZ_2$.
\item Then, Lemma~\ref{lemma:bijection-subgroup-graph} translates the question of finding a free subgroup $H$ of $T$ into finding a non-degenerate $(T,X)$-graph $\cG$.
\item By Lemma~\ref{lemma:N-over-H-as-auto}, the condition that $N_T(H)\diagup H$ be isomorphic to ${\Gamma}$ is equivalent to the condition that the automorphism group of $\cG$ be isomorphic to ${\Gamma}$. Hence, the initial problem effectively reduces to finding a non-degenerate $(T,X)$-graph with a given automorphism group.
\item Starting with the Cayley graph $\Cay({\Gamma},S)$ of the finite group $\Gamma$, we replace its edges and vertices by certain pieces of non-degenerate $(T,X)$-graphs (defined in Section~\ref{sssection:substitution}), so that the automorphism group is preserved, while obtaining a valid Schreier graph for a finite-index free subgroup $H$ of $T$. 
\end{enumerate}

\subsection{Basic cases} \label{section:basiccases}

We start first by proving that $\ZZ_p*\ZZ_q$, with $p  \geq  3$, $q  \geq  2$, and $\ZZ_2*\ZZ_2*\ZZ_2$ are freely telescopic. These are the ``base cases'' for the general statement that follows in Theorem~\ref{thm:free-products-are-telescopic}.

\begin{prop}\label{prop:ZpZq-telescopic}
The free product $\ZZ_p \ast \ZZ_q$ is freely telescopic for any $p  \geq  3$ and $q  \geq  2$.
\end{prop}

\begin{prop}\label{prop:Z2Z2Z2-telescopic}
The free product $\ZZ_2 \ast \ZZ_2 \ast \ZZ_2$ is freely telescopic.
\end{prop}

The proofs of these two results rely on a LEGO-like construction using pieces of non-degenerate $(T,X)$-graphs, that we produce below.
Once this is done, and the necessary properties of the construction hold, the proofs will follow easily, cf. Section~\ref{sssection:proof-of-telescopicity}.

\subsubsection{Vertex splitting and gluing}
Let us consider a Schreier graph $\cS = \Sch_{T,X}(H)$ of a subgroup $H \leq T$, where $T$ has generating set $X$. If $Y \subsetneq X$, a vertex $v$ of $\cS$ is \emph{split} along $Y$ if $v$ is replaced by two vertices $v_Y,v_{X-Y}$, where $v_Y$ keeps the $Y$-coloured edges of $v$, and $v_{X-Y}$ keeps its $(X-Y)$-coloured edges, as shown in Figure~\ref{fig:vertex-splitting}. We shall call $v_Y$ a \textit{dangling} $Y$-coloured vertex. Observe that splitting vertices breaks the $X$-regularity of the graph. If $u_Y$ and $v_{X-Y}$ are, respectively, $Y$- and $(X-Y)$-coloured dangling vertices we say that $u_Y$ and $v_{X-Y}$ are \textit{complementary}, and were we to identify them, we would gain regularity back at the newly created vertex. With this idea in mind, dangling vertices are seen as ``connection points'' for our graphs: an $Y$-coloured dangling vertex can only be connected to an $(X-Y)$-dangling vertex, and once all dangling vertices of a graph are connected, the resulting graph is $X$-regular. Let us call the identification of complementary vertices \emph{gluing}.

Finally, if $\cS$ is a Schreier graph of a free subgroup of $T$, it has no degenerate cycles. Since the operations of splitting vertices and gluing complementary ones do not change the lengths of cycles of any given colour, as soon as regularity is gained back by gluing all dangling vertices of some split graph, one gets the Schreier graph of a \emph{free} subgroup once again.

The following is now essentially obvious from the above considerations. 
\begin{prop}\label{prop:splitting-gluing}
Let us choose $Y \varsubsetneq X$ and a finite number of Schreier graphs $\cS_i = \Sch_{T,X}(H_i)$ of free, finite-index subgroups $H_i$ of $T$. Consider their disjoint union ${\sqcup}\cS_i$, in which we split a certain number of vertices along $Y$, and glue them, in complementary pairs, so that the resulting graph is connected. Then we obtain a Schreier graph for a free, finite-index subgroup of $T$.
\end{prop}

    \begin{figure}
        \centering
        \begin{tikzpicture}
            \node (v) at ( 0, 0) {$v$}; 
            \node (ro) at (-1, -1) {};
            \node (rt) at (-1,1) {};
            \node (bo) at (-1,-2) {};
            \node (bt) at (-1,0) {};
            \node (go) at (1,1) {};
            \node (gt) at (1,-1) {};

            \begin{scope}[every path/.style={->}]
                \draw[red] (ro) -- (v);
                \draw[blue] (bo) -- (v); 
                \draw[green] (go) -- (v);
                \draw[blue] (v) -- (bt);
                \draw[red] (v) -- (rt);
                \draw[green] (v) -- (gt);
            \end{scope}  
        \end{tikzpicture}
        \qquad
        \begin{tikzpicture}
            \node (v) at ( 0, 0) {$v_{rb}$}; 
            \node (u) at ( 1, 0) {$v_g$}; 
            \node (ro) at (-1, -1) {};
            \node (rt) at (-1,1) {};
            \node (bo) at (-1,-2) {};
            \node (bt) at (-1,0) {};
            \node (go) at (2,1) {};
            \node (gt) at (2,-1) {};

            \begin{scope}[every path/.style={->}]
                \draw[red] (ro) -- (v);
                \draw[blue] (bo) -- (v); 
                \draw[green] (go) -- (u);
                \draw[blue] (v) -- (bt);
                \draw[red] (v) -- (rt);
                \draw[green] (u) -- (gt);
            \end{scope}  
        \end{tikzpicture}
        \caption{Before and after splitting an $X$-regular graph, with $X = \{\text{\textbf{r}ed}, \text{\textbf{g}reen}, \text{\textbf{b}lue}\}$, at a vertex $v$ along $Y = \{r,b\}$}.
        \label{fig:vertex-splitting}
    \end{figure}
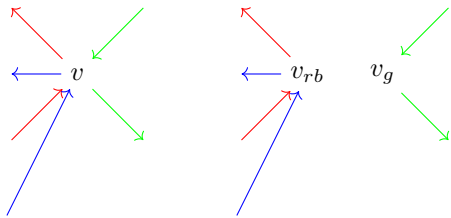
    
\subsubsection{Sketch of the construction}
Below we explain the main idea of the construction.  Let $T$ be either $\ZZ_p*\ZZ_q$ or $\ZZ_2*\ZZ_2*\ZZ_2$, with $X$ its natural set of generators, and ${\Gamma}$ be a finite group generated by a set $S$. Let us choose two $X$-coloured graphs, say $\cL_e$ and $\cL_v$ obtained by splitting, respectively, one and two vertices in the Schreier graph of a finite-index free subgroup of $T$, i.e. a non-degenerate $(T,X)$-graph. Call $\cL_e$ an \emph{edge-link} and $\cL_v$ a \emph{vertex-link}. We shall also choose an easily identifiable and \textit{unique}, as we shall see in the sequel, vertex of $\cL_e$ to be its \textit{root}, denoted $r(\cL_e)$.

We shall connect edge-links by gluing complementary vertices, so as to connect them into chains, and call the result \emph{edge-graphs}. Similarly, we shall connect vertex-links but in a way to produce cycles of them, and call the result \emph{vertex-graphs}. Finally, in the Cayley graph $\cC:= \Cay({\Gamma},S)$ of ${\Gamma}$, we shall replace the edges by edge-graphs and the vertices by vertex-graphs, following the procedure of Section~\ref{sssection:substitution}. Our construction will ensure that
\begin{itemize}
\item[(i)] the automorphism group $\aut \cC^*$ of the resulting graph $\cC^*$ is the same as the automorphism group of $\cC$, which is exactly ${\Gamma}$ (by Proposition~\ref{prop:graph-substitution}); and 
\item[(ii)] the graph $\cC^*$ is actually the Schreier graph of a finite-index free subgroup of $T$ (by Proposition~\ref{prop:splitting-gluing}).
\end{itemize}
    
From now on, if $v$ is a vertex of a \emph{folded} $X$-coloured graph, we write $v\cdot x_{1} \dots x_n$ for the terminus of the unique path labelled $x_{1}, \dots ,x_n$ and starting at $v$, if it exists. Then, an equality of the form $v\cdot x_{1} \dots x_n = u\cdot y_{1} \dots y_n$ holds if and only if  both paths in question exist and their termini are equal.  

\subsubsection{Constructing the links}\label{section:constructionofthelinks}
        
Our construction differs slightly for $T = \ZZ_p*\ZZ_q$ and $\ZZ_2*\ZZ_2*\ZZ_2$. In the former case, one uses a relatively generic construction, while the latter is mostly ad-hoc.

\medskip

\paragraph{\textbf{Case} $T = \ZZ_p*\ZZ_q$} Let us fix any  $p \geq 3$ and $q \geq 2$, let the corresponding generators of each free factor of $T$ be \textbf{r}ed and \textbf{c}yan and consider $q$ copies of a $p$-cycle with edge labels $r$ and vertices $v_{0,i}, \dots, v_{p-1,i}$ for each $i$-th copy, where $0{\leq}i{\leq}q-1$. Next, add two ``special'' $q$-cycles labelled $c$:
\[
v_{0,0} {\cto} v_{1,0} {\cto} v_{1,1} {\cto} v_{1,2} {\cto} \dots {\cto} v_{1,q-2} {\cto} v_{0,0},
\]
and
\[
v_{0,1} {\cto}  v_{0,2} {\cto}  \dots {\cto} v_{0,q-1} {\cto} v_{1,q-1} {\cto} v_{0,1}.
\]
Then, for each $j \geq 2$, draw an extra $q$-cycle labelled $c$:
\[
v_{j,0} {\cto} v_{j,1} {\cto} \dots {\cto} v_{j,q-1} {\cto} v_{j,0}.
\]
In the case $q=2$, the extra $q$-cycles have the form $v_{0,0}{\cto} v_{1,0} {\cto} v_{0,0}$ and $v_{0,q-1}{\cto} v_{1,q-1} {\cto} \ldots {\cto} v_{0,q-1}$. Let $\cG_{p,q}$ denote the resulting non-degenerate $(\ZZ_p*\ZZ_q, \{r,c\})$-graph.

Now, split the vertex $v_{0,0}$ in order to produce an \textit{edge-link} denoted $\cL_{p,q}^e$ and, subsequently, split the vertex $v_{0,q-1}$ to get a \textit{vertex-link} $\cL_{p,q}^v$. The vertices obtained by splitting $v_{0,0}$ will be denoted $v_{+}, v_{-}$, and those obtained by splitting $v_{0,q-1}$ will be called $u_+, u_-$. In the sequel, we swap the assignment of $v_+$ and $v_-$ vertices for edge- and vertex-links, as shown in Figures~\ref{fig:Link-Z3Z2}-\ref{fig:Link-Z2Z2Z2}. This allows us to keep a consistent and clear notation for all associated objects. 

In the edge-link $\cL_{p,q}^e$, the vertex $v_{0,q-1}$ is unique in the following sense: this is the only vertex $v \in V\cL_{p,q}^e$ that satisfies $v\cdot r=v\cdot c$. Observe that the assumption $p \geq 3$ is important here: if $p=q=2$, then once $v\cdot r=v\cdot c$, the vertex $w=v\cdot r$ also satisfies $w\cdot r = w\cdot c$.  Let then $r(\cL_{p,q}^e) := v_{0,q-1}$ be called the \textit{root}\footnote{In the case of an edge- or vertex-link the notion of a root is practically opposed to the notion of a root in a graph (map, hypermap, etc.). Indeed, the former is intrinsic to the respective combinatorial structure, while the latter is a matter of choice and can be assigned arbitrarily} of $\cL_{p,q}^e$.

\begin{figure}[h]
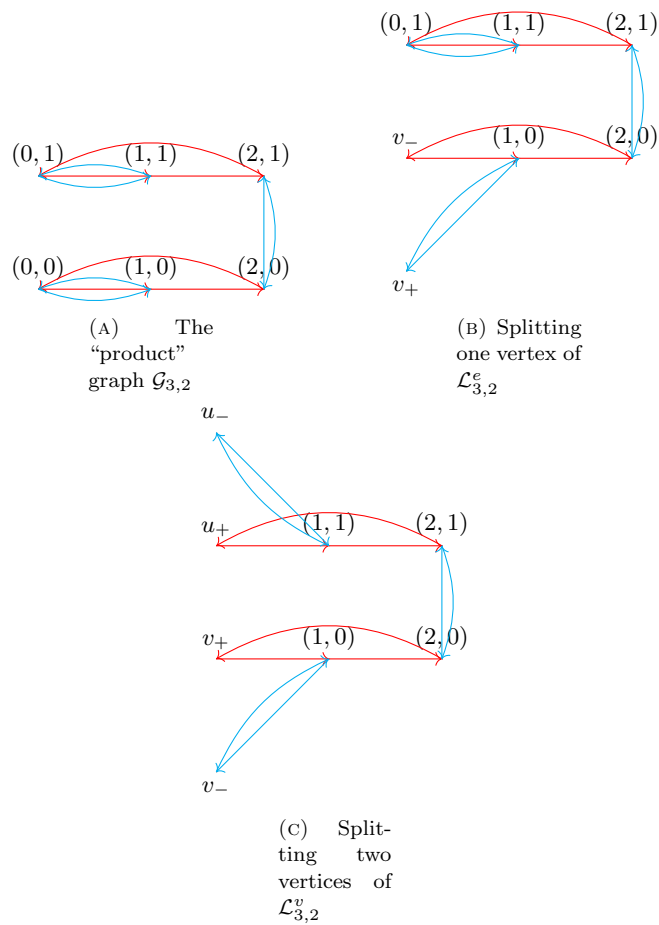

\centering
\subfloat[The ``product'' graph $\cG_{3,2}$]{\tikzproductgraph{2}{3}{0}}
\qquad
\subfloat[Splitting one vertex of $\cL_{3,2}^e$]{\tikzproductgraph{2}{3}{1}}
\qquad
\subfloat[Splitting two vertices of $\cL_{3,2}^v$]{\tikzproductgraph{2}{3}{2}}
\caption{The ``link'' graphs for $T = \ZZ_3*\ZZ_2$: (a) the original graph; (b) the result of splitting $(0,0)$; (c) the result of splitting $(0,0)$ and $(0,1)$. The root of the edge link is $(0,1)$.}
\label{fig:Link-Z3Z2}
\end{figure}

Some examples of ``product graphs'' and their splitting at one and two vertices that generate edge- and vertex-link graphs are depicted in Figures~\ref{fig:Link-Z3Z2}--\ref{fig:Link-Z4Z4}.

\begin{figure}[h]
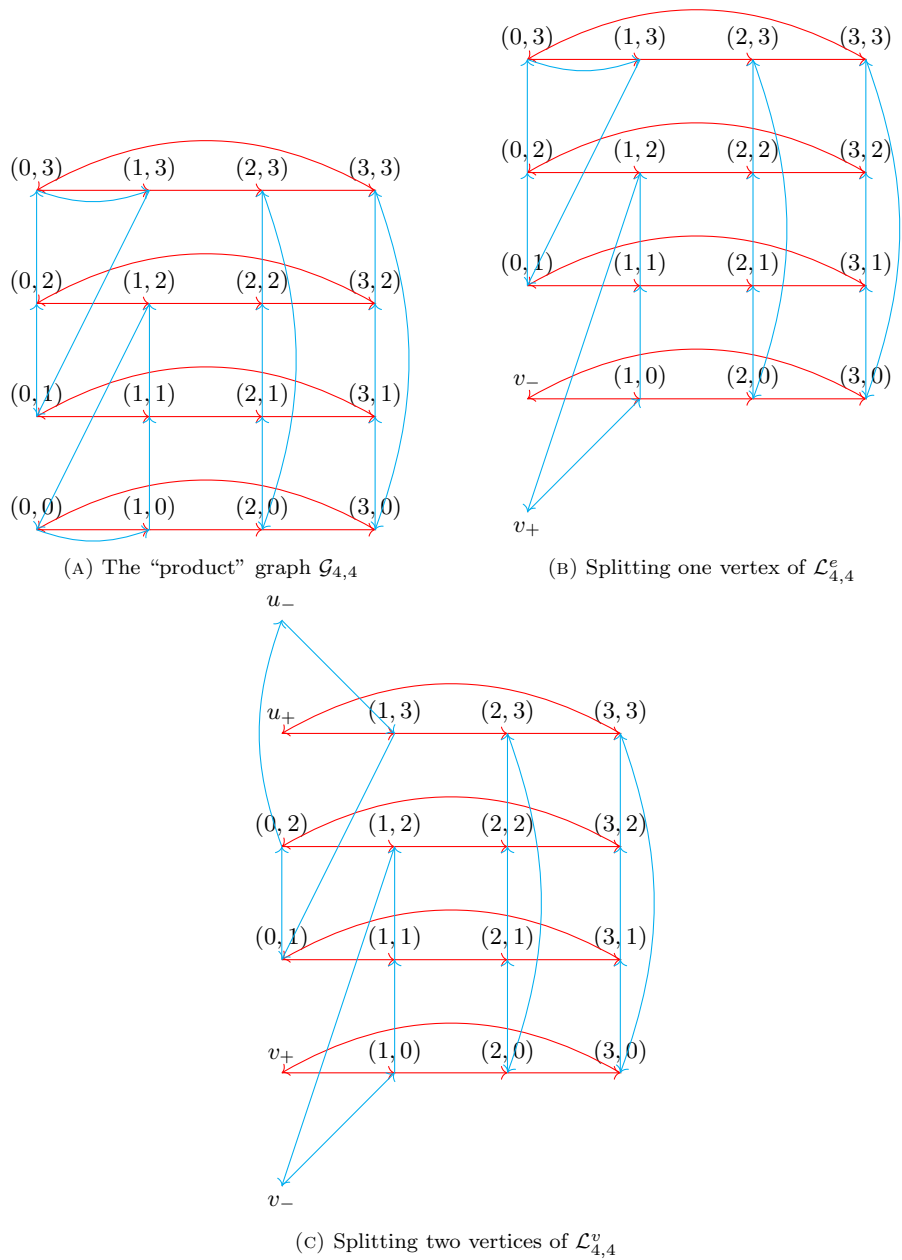

\centering
\subfloat[The ``product'' graph $\cG_{4,4}$]{\tikzproductgraph{4}{4}{0}}
\qquad
\subfloat[Splitting one vertex of $\cL_{4,4}^e$]{\tikzproductgraph{4}{4}{1}}
\qquad
\subfloat[Splitting two vertices of $\cL_{4,4}^v$]{\tikzproductgraph{4}{4}{2}}
\caption{The ``link'' graphs for $T = \ZZ_4*\ZZ_4$: (a) the original graph; (b) the result of splitting $(0,0)$; (c) the result of splitting $(0,0)$ and $(0,3)$. The root of the edge link is $(0,3)$.}
\label{fig:Link-Z4Z4}
\end{figure}

\medskip

\paragraph{\textbf{Case} $T = \ZZ_2*\ZZ_2*\ZZ_2$} Let the generators of the group $T$ be \textbf{r}ed, \textbf{g}reen and \textbf{b}lue. Let $\cG_{2,2,2}$ denote the graph (a) in Figure~\ref{fig:Link-Z2Z2Z2}. In $\cG_{2,2,2}$, first split the vertex $v_7$ to get an edge-link that we call $\cL_{2,2,2}^e$, and then split $v_0$ to get a vertex-link called $\cL_{2,2,2}^v$. The vertices obtained by splitting the vertex $v_7$ will be denoted $v_+, v_-$, and those obtained by splitting the vertex $v_0$ will be called $u_+, u_-$.
        
Observe that the vertices $v_0$ and $v_1$ are unique in $\cL_{2,2,2}^e$ in the following sense: they are the only vertices $v \in V\cL_{2,2,2}^e$ satisfying $v\cdot b=v\cdot g$. Furthermore, one can distinguish $v_0$ from $v_1$ as follows: while $v_{1}\cdot rgbgrgr = v_{1}$, it is not the case for $v_0$. In other words, the path labelled $rgbgrgr$ and starting at $v_{1}$ is a loop, but the one identically labelled and starting at $v_{0}$ is not. Let then $r(\cL_{2,2,2}^e) := v_{1}$ be the \textit{root} of $\cL_{2,2,2}^e$.

        \begin{figure}
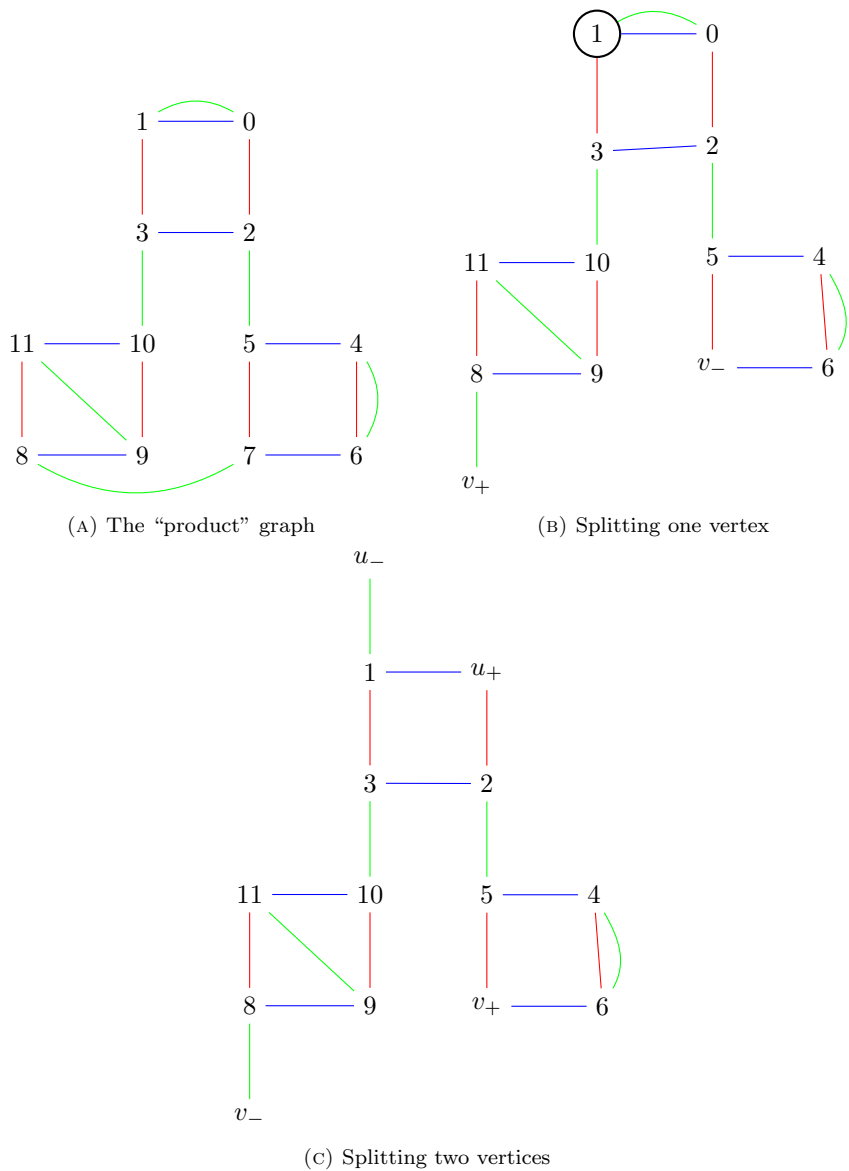

            \centering
            \subfloat[The ``product'' graph]{\tikztttzarms}
            \qquad
            \subfloat[Splitting one vertex]{\tikzttttarms}
            \qquad
            \subfloat[Splitting two vertices]{\tikztttfarms}
            \caption{The ``link'' graphs for $T = \ZZ_2*\ZZ_2*\ZZ_2$: (a) the original graph; (b) the result of splitting $7$; (c) the result of splitting $0$ and $7$. The root of the edge link is encircled.}
            \label{fig:Link-Z2Z2Z2}
        \end{figure}

Let us write $\cL^v$ for any of $\cL^v_{p,q}$ and $\cL^v_{2,2,2}$, and $\cL^e$ for any of $\cL^e_{p,q}$ and $\cL^e_{2,2,2}$. Since the following constructions do not depend on the exact nature of those graphs, but rather on their abstract properties, this ambiguity is harmless.

\medskip

\paragraph{\textbf{On the behaviour of roots}} Let us define the following conditions
\begin{alignat*}{10}
P_{p,q}(v) &:= \text{$``v\in V\cL_{p,q}^e$ satisfies $v\cdot r = v\cdot c$''},\\
P_{2,2,2}(v) &:= \text{``$v\in V\cL_{2,2,2}^e$ satisfies $v\cdot rgbgrgr$ = $v$ and $v\cdot b = v\cdot g$''},
\end{alignat*}
and observe that the following statements hold for $P = P_{p,q}$ or $P_{2,2,2}$, whichever is appropriate.

\begin{enumerate} 
\renewcommand{\labelenumi}{\textbf{G.\theenumi}}
\item \label{ppty:G1} No vertex of a vertex-link $\cL^v$ satisfies $P$, and exactly one vertex $r(\cL^e)$ of an edge-link $\cL^e$ does.
\item \label{ppty:G2} Gluing edge- and vertex-links together by identifying complementary dangling vertices does not create new vertices satisfying $P$, as long as the gluing is done on vertices with disjoint neighbourhoods (where the neighbourhood of a vertex $v$ is the set of adjacent vertices.)
\item \label{ppty:G3} If ${\iota}:\cA {\hookrightarrow} \cB$ is an embedding of folded graphs (in our case $\cA$ and $\cB$ will be obtained by gluing vertex- and edge-links), then the image of a vertex satisfying $P$ also satisfies $P$. 
\end{enumerate}

Of the above, G.\ref{ppty:G1} holds by construction of the vertex- and edge-links and G.\ref{ppty:G3} is evident from the fact that ``following the labels'' and ``passing to the image under an embedding'' are commuting operations. Only G.\ref{ppty:G2} is not as direct. Let $u_Y$ and $v_{X-Y}$ be complementary dangling vertices with disjoint neighbourhoods. Then, letting $x$ be a label in $X-Y$, $y$ be a label in $Y$,  and $w$ be the result of gluing $u_Y$ to $v_{X-Y}$, we get $w\cdot x = v_{X-Y} \cdot x$ and $w\cdot y = u_Y\cdot y$, which are distinct by the hypothesis. Since both conditions $P_{p,q}$ and $P_{2,2,2}$ involve equalities of the form $v\cdot x = v\cdot y$, it follows that no glued complementary vertices can satisfy them.

\subsubsection{Constructing vertex- and edge-graphs}\label{sssection:edge-and-vertex-graphs}

First of all, let us introduce some necessary notation, which will also be used in Section~\ref{sssection:substitution}, later on.
        
A \emph{vertex-graph} is a graph $\cV$, along with an injection ${\chi}:S{\times}\{+,-\}{\hookrightarrow} V\cV$, an example of which is depicted in Figure~\ref{fig:vertex-graph-5gen}. Let the \textit{boundary} of a vertex-graph be $\partial \cV = \im \chi$. 

An \emph{edge-graph} for a label $s$ is a graph $\cE_s$, along with two distinguished vertices $h^{+}(\cE_s)$ and $h^{-}(\cE_s)$, as shown in Figure~\ref{fig:edge-graph}. We shall provide the general definitions of the graphs $\cV$ and $\cE_s$ below, which will be case-specific for different choices of $T$. 

\begin{figure}[ht]
\centering
        \begin{tikzpicture}
            \pgfmathtruncatemacro{\len}{5};
            \pgfmathtruncatemacro{\width}{2};
            \pgfmathsetmacro{\shift}{0.7};
            \pgfmathtruncatemacro{\baseangle}{360/(\len+1)};
            \foreach \i in {0,...,\len}
            {
                \pgfmathsetmacro{\left}{\shift};
                \pgfmathtruncatemacro{\ip}{\i+1};
                \pgfmathsetmacro{\cent}{\shift+(\width/2)};
                \pgfmathsetmacro{\right}{\shift+\width};
                \pgfmathsetmacro{\rot}{\i*\baseangle};
                \begin{scope}[rotate=\rot]
                    \draw[rounded corners=0.5cm] (\left,-1) rectangle  (\right,1);
                    \draw[fill, red] (\right,0.5) circle [radius=2pt];
                    \node at (1.2*\right, \left) {$u_+^{(\ip)}$};
                    \draw[fill, cyan] (\right,-0.5) circle [radius=2pt];
                    \node at (1.2*\right,-0.5) {$v_-^{(\ip)}$};
                    \draw[fill, gray] (\cent,-1) circle [radius=2pt];
                    \draw[fill, gray] (\cent,1) circle [radius=2pt];
                \end{scope}
            }
        \end{tikzpicture}

\caption{The vertex-graph $\cV$ for $|S|=6$. The squares represent the connected vertex-links $\cL^v$, the gray dots represent the vertices glued along the way, while the cyan and red dots represent the remaining dangling vertices, which are exactly the images of $\chi$.}\label{fig:vertex-graph-5gen}
\end{figure}
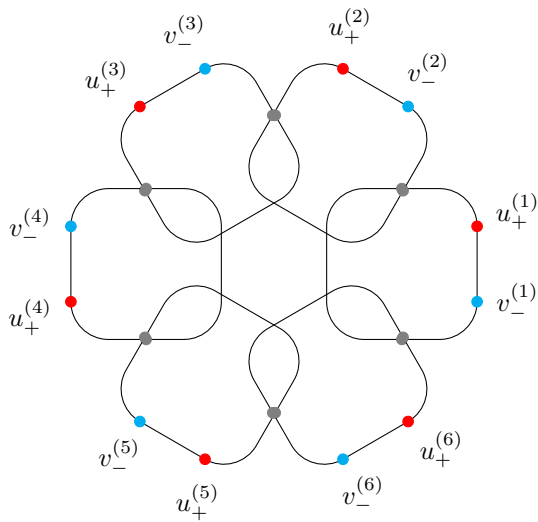

In the Cayley graph $\cC := \Cay({\Gamma},S)$ of ${\Gamma}$, each vertex $v$ will be replaced by an isomorphic copy $\cV_v$ of $\cV$, and each edge $e$ labelled $s$ will be replaced by an isomorphic copy $\cE_e$ of $\cE_s$. We shall also make the following identifications:
        \[
            \cE_e \ni h^{+}(\cE_s) \sim {\chi}(s,+)  \in  \cV_u, \qquad \cE_e \ni h^{-}(\cE_s) \sim {\chi}(s,-)  \in  \cV_v,
        \]
whenever $e$ has label $s$, origin $u$ and terminus $v$, where $h^{\pm}(\cE_s) {\in} \cE_e$ means a copy of the vertex $h^{\pm}(\cE_s)$ inside $\cE_e$, and similarly for other instances of vertex- and edge-graphs. 

Let $\cC^*$ denote the resulting graph, and let ${\iota}_e:\cE_s {\hookrightarrow} \cC^*$ and ${\iota}_v:\cV {\hookrightarrow} \cC^*$ be the embeddings corresponding to an edge $e \in V\cC$ (with $\mu(e) = s$) and a vertex $v\in V\cC$, respectively. The image ${\iota}_e(\cE_s)$ will be denoted by $\cE_e$ and the image ${\iota}_v(\cV)$ will be called $\cV_v$. Let $h^{\pm}(e)$ be the image of $h^{\pm}(\cE_s)$ under ${\iota}_e$.

Finally, if an ordering on the labels $s_{1} < \dots < s_n$ is given, then let 
\[
    \cC^*_{(i)} := \left(\bigcup_{v{\in}V\cC} \cV_v \right) \cup \left(\bigcup_{j{\leq}i} \bigcup_{\substack{e{\in} E\cC\ \\ \text{s.t.} {\mu}(e)=s_j}} \cE_e \right),
\]
which means that $\cC^*_{(i)}$ is the subgraph of $\cC^*$ consisting of the vertex-graphs and \textit{only} the edge-graphs corresponding to the edges with labels $s_j$ for $j{\leq}i$.

From now on, assume that the generating set $S = \{s_{1}, \dots, s_n\}$ is ordered: $s_1 < s_2 < \dots < s_n$. First, we consider the vertex- and edge-graphs that we use in the case $T = \mathbb{Z}_p*\mathbb{Z}_q$, with $p \geq q$, $p\geq 3$, $q\geq 2$. 

\medskip

\paragraph{\textbf{Vertex-graphs}} 
Let us take $|S|$ disjoint copies of the vertex-link $\cL^{v}$ from Section~\ref{section:constructionofthelinks}, and call them $\cL_v ^{(0)}, \dots, \cL_v ^{(|S|-1)}$, respectively. Observe that each graph $\cL_v ^{(i)}$ has four dangling vertices $v^{(i)}_+,v^{(i)}_-$, $u^{(i)}_+,u^{(i)}_-$. In order to create the vertex graph $\cV$, identify each $v^{(i)}_+$ to $u^{((i+1) \mod |S|)}_-$, for $0{\leq}i{\leq}|S|-1$. Let us also define $\chi : S{\times}\{+,-\} \hookrightarrow V\cV$ as $\chi((s_i,+))=u_+^{(i)}$ and $\chi((s_i,-))=v_-^{(i)}$. For a sketch of the resulting graph, see Figure~\ref{fig:vertex-graph-5gen}. 
        
Note that in the case $|S|=1$, only one vertex-link is used, and the fact that $v^{(0)}_{+}$ is glued to $u^{(0)}_{-}$, and not to $v^{(0)}_{-}$, ensures that the hypothesis G.\ref{ppty:G2} holds, so that no extra root appears.

\begin{figure}[ht]
\centering
        \begin{tikzpicture}
            \foreach \i in {1,...,4}
            {
                \pgfmathtruncatemacro{\left}{2*\i};
                \pgfmathtruncatemacro{\cent}{(2*\i)+1};
                \pgfmathtruncatemacro{\right}{(2*\i)+2};
                \draw[rounded corners=0.5cm] (\left,0) rectangle (\right,2);
                \draw[fill, gray] (\left,1) circle [radius=2pt];
                \node[below right] at (\left,1) {$v_+^{(\i)}$};
                \draw[fill, gray] (\right,1) circle [radius=2pt];
                \node[below left] at (\right,1) {$v_-^{(\i)}$};
            }
        \draw[fill, cyan] (2,1) circle [radius=2pt];
        \draw[fill, red] (10,1) circle [radius=2pt];
		\node[below left] at (2,1) {$h^+(e)$};        
        \node[below right] at (10,1) {$h^-(e)$};  
        \end{tikzpicture}
\caption{An edge-graph of length $4$. The squares represent the connected edge-links $\cL^e$, and the cyan and red dots represent the remaining dangling vertices, which are exactly $h^{+}(e)$ and $h^{-}(e)$.}\label{fig:edge-graph}
\end{figure}
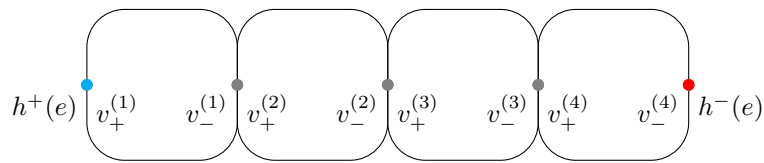

\medskip

\paragraph{\textbf{Edge-graphs}}
For each $s_i{\in}S$, take $i$ copies of the edge-link $\cL^{e}$ from Section~\ref{section:constructionofthelinks}, and call them $\cL_e ^{(1)}, \dots ,\cL_e ^{(i)}$, respectively. Observe that each graph $\cL_e ^{(j)}$ has two dangling vertices $v^{(j)}_+,v^{(j)}_-$. Now, glue each $v^{(j)}_{-}$ to $v^{(j+1)}_{+}$, for $1 \leq j \leq i-1$. Let $\cE_i$ denote the resulting edge-graph, a sketch of which is depicted in Figure~\ref{fig:edge-graph}. The vertex $v^{(1)}_+$ will be referred to as $h^{+}(e)$ and $v^{(i)}_-$ as $h^{-}(e)$.

Note that the exact same construction works for $T=\ZZ_2*\ZZ_2*\ZZ_2$, when using the corresponding objects, i.e. when $\cL^v_{p,q}$ is replaced by $\cL^v_{2,2,2}$ and $\cL^e_{p,q}$ is replaced by $\cL^e_{2,2,2}$. 

In both cases, letting $P$ be the appropriate condition, i.e. either $P_{p,q}$ or $P_{2,2,2}$, each edge graph $\cE_i$ contains exactly $i$ vertices satisfying $P$: they are exactly the copies of the root vertex $r(\cL^e)$ satisfying $P$ in the edge-link $\cL^e$.  Let $r_1(s_i), \dots ,r_i(s_i)$, where $r_{j}(s_i) {\in} \cL_{e}^{(j)}$, denote those roots. Observe that the vertex-graphs do not contain any vertex satisfying $P$ and that $\aut(\cE_i) = \{\mathrm{id}\}$, where the former follows by construction and the latter follows from foldedness. 

\subsubsection{Substitution}\label{sssection:substitution-works}

Let $\cC=\Cay({\Gamma},S)$ be the Cayley graph of a finite group ${\Gamma}$ with respect to a generating set $S= \{s_1, s_2, \dots, s_n\}$. Our goal is to prove that $\aut(\cC^{*}) \cong \aut(\cC)$ which we can do by using the argument of Proposition~\ref{prop:graph-substitution} (cf. Section~\ref{graph-substitution} for its proof). Thus, we only have to show that the following conditions hold:

\begin{description}
\item[S.\ref{S1}] $\{ \phi {\in} \aut \cV :  {\phi}|_{ {\partial}\cV } = \mathrm{id}_{{\partial} \cV} \} = \{ \mathrm{id}_{\cV} \}$;
\item[S.\ref{S2}] one can order the labels $s_{1}, \dots, s_n$  in such a way that for any $i\in\{0, \dots ,n\}$, $\cC^*_{(i)}$ contains no subgraph isomorphic to $\cE_{s_{i}}$ except for the subgraphs $\cE_{e}$, for $e\in E\cC,{\mu}(e)=s_{i}$.
\end{description}

In the edge-link $\cL^e$, choose a shortest path from $v_{+}$ to $r(\cL^e)$ and from $r(\cL^e)$ to $v_{-}$, and let $w_{1}$ and $w_{2}$ be their respective labels.

\begin{description}
\item[S.\ref{S1}] Recall that $\cV$ is a folded graph. Let ${\phi} \in \aut \cV$ be an automorphism with ${\phi}|_{{\partial}\cV} = \mathrm{id}_{{\partial}\cV}$, and let $v\in{\partial}\cV$ be a vertex in the boundary of $\cV$. Then, since ${\phi}$ and $\mathrm{id}_{\cV}$ agree at $v$, and $\cV$ is folded, we have ${\phi}= \mathrm{id}_{\cV}$. Thus \textbf{S.\ref{S1}} holds.
\item[S.\ref{S2}]  Fix $i \in \{ 0, \dots, n\}$ and consider $\cC^*_{(i)}$.  Our goal is to show that $\cC^*_{(i)}$ contains no other copies of $\cE_i$ than those of the form $\cE_e$, for $e$ an edge with label $\mu(e) = s_i$. Let $r:=r_1(s_i)$ be the root of the first edge-link in $\cE_i$. Suppose that there exists an embedding $\iota : \cE_i \hookrightarrow \cC^{*}_{(i)}$, and let us show that $\iota(\cE_i)$ is equal to ${\iota}_e(\cE_i)$ for some $e$ with label $\mu(e) = s_i$.

We know that $\iota(r)$ lies in an edge-graph because $\iota(r)$ satisfies property $P$ by~G.\ref{ppty:G3} and no vertex of a vertex-graph satisfies property $P$. Let $e\in \cE_e$ be the edge such that $\iota(r)$ lies in $\cE_{e}$, and $s_k := {\mu}(e)$. It is essential to notice that in $\cE_i$, for any $1{\leq}j{\leq}i-1$, we have $r_j(s_i)w_{2}w_{1} = r_{j+1}(s_i)$. Similarly, in $\cE_e$, for any $1{\leq}j{\leq}k-1$, it holds that $r_j(e)w_{2}w_{1} = r_{j+1}(e)$. A sketch of how the roots and boundary vertices are arranged within an edge graph is given in Figure~\ref{fig:edge-graph-embedding}. 

We shall show that $e$ has label $s_i$ and ${\iota}(r) = r_1(e)$. This is enough to conclude that ${\iota}(\cE_i)$ equals ${\iota}_e(\cE_i)$, since, by assumption, the embeddings ${\iota}$ and ${\iota}_{e}$ agree at a vertex $r \in \cE_i$, and thus must coincide by foldedness.

To this end, assume that $e$ has label $s_k$ with $k{\leq}i$, and ${\iota}(r) = r_j(e)$, necessarily with $1{\leq}j{\leq}k$. Then, we prove that $k-j \geq i-1$, which is equivalent, by using the constraints $k{\leq}i$ and $1{\leq}j{\leq}k$, to $k=i$ and $j=1$. 

Let us assume, contrary to the above, that $k-j<i-1$. Following the path labelled $(w_{2}w_{1})^{k-j+1}$ and starting at $r = r_{1}(s_i)$ in $\cE_i$, we reach the vertex $r_{k-j+2}(s_i)$ in $\cE_i$, which is a root. We claim that $r_j(e)(w_{2}w_{1})^{k-j+1}$ is \emph{not} a root, which thus contradicts ${\iota}$ being an embedding, since $r_j(e)(w_{2}w_{1})^{k-j+1} = {\iota}(r_{k-j+2}(s_i))$ is the image of a root.
                
First, let us consider the vertex $r_j(e)(w_{2}w_{1})^{k-j} = r_k(e)$. This is the last root of $\cE_e$, so that $r_k(e)w_{2}$ is in a vertex-graph: in fact, $r_k(e)w_{2}=h^-(e)$.  Now, if $\widehat{r} := r_j(e)(w_{2}w_{1})^{k-j+1} = r_k(e)(w_{2}w_{1})$ were a root somewhere in $\cC^*_{(i)}$, then $r_k(e)w_{2} = \widehat{r}w_{1}^{-1}$ would either be in the interior of an edge-graph, or of the form $h^+(e')$ for some edge $e'$. Indeed, if $\widehat{r}$ were the first root $r_1(e')$ of some $\cE_{e'}$, then $r_1(e')w_{1}^{-1}  = h^+(e')$ would hold. If it were a subsequent root, then $\widehat{r}w_{1}^{-1}$ would lie in the interior of $\cE_{e'}$. In neither case can it be equal to $r_k(e)w_{2} = h^-(e)$, and the desired contradiction is reached.
\end{description}

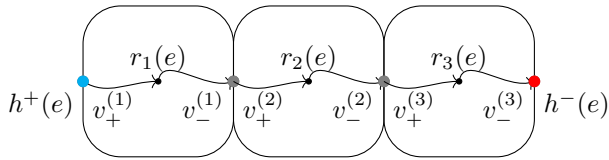
\begin{figure}[ht]
\centering
        \begin{tikzpicture}
            \foreach \i in {1,...,3}
            {
                \pgfmathtruncatemacro{\left}{2*\i};
                \pgfmathtruncatemacro{\cent}{(2*\i)+1};
                \pgfmathtruncatemacro{\right}{(2*\i)+2};
                \draw[rounded corners=0.5cm] (\left,0) rectangle (\right,2);
                \draw[fill] (\cent,1) circle [radius=1pt];
                \node[above] at (\cent,1) {$r_\i(e)$};
                \draw[fill, gray] (\left,1) circle [radius=2pt];
                \node[below right] at (\left,1) {$v_+^{(\i)}$};
                \draw[fill, gray] (\right,1) circle [radius=2pt];
                \node[below left] at (\right,1) {$v_-^{(\i)}$};
                \draw[->] (\left,1) to [out=-32,in=180] (\cent,1);
                \draw[->] (\cent,1) to [out=82,in=200] (\right,1);
            }
            \draw[fill, cyan] (2,1) circle [radius=2pt];
            \draw[fill, red] (8,1) circle [radius=2pt];
		    \node[below left] at (2,1) {$h^+(e)$};        
            \node[below right] at (8,1) {$h^-(e)$};  
        \end{tikzpicture}
\caption{Edge-links assembled into an edge-graph, their respective roots and boundary vertices.}\label{fig:edge-graph-embedding}
\end{figure}

\subsubsection{Free telescopicity of $\ZZ_p*\ZZ_q$ and $\ZZ_{2}*\ZZ_{2}*\ZZ_{2}$}\label{sssection:proof-of-telescopicity}

\begin{proof}[Proof of Propositions~\ref{prop:ZpZq-telescopic} and~\ref{prop:Z2Z2Z2-telescopic}] Given $T$ either $\ZZ_p*\ZZ_q$ ($p \geq q$, $p\geq 3$, $q\geq 2$) or $\ZZ_2*\ZZ_2*\ZZ_2$ with its natural generating set $X$ (one generator for each cyclic factor), and ${\Gamma}$ a finite group with generators $S$, let us substitute the edges and vertices of $\cC := \Cay({\Gamma},S)$ by the vertex- and edge-graphs as described in Sections \ref{sssection:edge-and-vertex-graphs} -- \ref{sssection:substitution-works}. 

By Proposition~\ref{prop:splitting-gluing}, the resulting graph $\cC^*$ is a non-degenerate $(T,X)$-graph with automorphism group
\[
\aut \cC^* {\cong} \aut \cC {\cong} {\Gamma},
\]
as verified in Section~\ref{sssection:substitution-works}. 
        
Let $H$ be a finite-index free subgroup of $T$ with $\Sch_{T,X}(H) {\cong} \cC^*$. Then, by Lemma~\ref{lemma:N-over-H-as-auto},
\[
N_T(H)/H \cong \Gamma.
\]
Since the procedure above can be performed for arbitrary ${\Gamma}$, both $\ZZ_p*\ZZ_q$ and $\ZZ_2*\ZZ_2*\ZZ_2$ are freely telescopic.
\end{proof}

\subsection{Adding factors}\label{section:addingfactors}

The following lemma allows us to pull back telescopicity by surjections. The only property not guaranteed in this case is freeness.

\begin{lemm}\label{lemma:surjection-telescopic}
Let $ f : A \longrightarrow B$ be an epimorphism. Then, if $B$ is telescopic, $A$ is also telescopic, although not necessarily freely.
\end{lemm}

\begin{proof}
Let $\Gamma$ be a finite group. Since $B$ is telescopic, there exists a finite index subgroup $H \leq B$ such that $\sfrac{N_B (H)}{H} \cong \Gamma$. Consider $ \widetilde{H}:= f^{-1}(H) \leq A$, the preimage of $H$ in $A$. Clearly, $\widetilde{H}$ is a finite index subgroup of $A$ since it is a preimage of a finite index subgroup under an epimorphism.

We have that $ N_A(\widetilde{H})=f^{-1}(N_B(H))$. Indeed, let $x$ be an element of $f^{-1}(N_B(H))$. By definition, this means that $f(x) \in N_B (H)$, and thus $f(x) H = H f(x)$. By applying $f^{-1}$ to both sides, we obtain $x \widetilde{H} = \widetilde{H} x$. The latter yields that $x$ is an element of $N_A(\widetilde{H})$. The reverse inclusion is analogous. 

Thus, $\sfrac{N_A (\widetilde{H})}{\widetilde{H}}$ equals $ \sfrac{f^{-1}(N_B(H))}{f^{-1}(H)}$. It remains to check that the quotient group $ \sfrac{f^{-1}(N_B(H))}{f^{-1}(H)}$ is isomorphic to $\sfrac{N_B(H)}{H} \cong \Gamma$. Let us consider the map $\widetilde{f} : f^{-1}(N_{B}(H)) \longrightarrow \sfrac{N_B(H)}{H}$ defined by $\widetilde{f}(x) = f(x) \cdot H$. Since $\widetilde{f}$ is a surjective homomorphism with kernel $\ker(\widetilde{f})=f^{-1}(H)$, the desired result follows from the first isomorphism theorem.
\end{proof}

The following lemmas are key to making an inductive step and proving the main result stated as Theorem~\ref{thm:free-products-are-telescopic}.

\begin{lemm}\label{lemma:telescopic-lcm}
Assume that a product of the form $T*\ZZ_{m}$, $m \geq 2$, is freely telescopic. Then $T*\ZZ_a*\ZZ_b$ is also freely telescopic, for all $a$, $b\geq 2$ such that $lcm(a, b) = m$.
\end{lemm}

\begin{proof}
Let $c := m/a$, $d := m/b$ and consider the following morphisms:
\begin{alignat*}{10}
\iota_a: \ZZ_a &\rightarrow \ZZ_m,\\
[k]_a &\mapsto \left[ ck \right]_m
\end{alignat*}
and        
\begin{alignat*}{10}
\iota_b: \ZZ_b &\rightarrow \ZZ_m\\
[k]_b &\mapsto \left[ dk \right]_m.
\end{alignat*}

Since $\iota_a([1]_a) = [d]_m$ and $\iota_b([1]_b) = [c]_m$, and $c$ and $d$ are coprime, any element of $\ZZ_m$ can be written as a sum of elements in the images of $\iota_a$ and $\iota_b$ by B\'ezout's theorem. Let us consider the surjective morphism $\phi: T*\ZZ_a*\ZZ_b {\to} T*\ZZ_m$ induced by the maps:
$$
\iota_a: \ZZ_a \rightarrow \ZZ_{m}, \quad \iota_b: \ZZ_b \rightarrow \ZZ_{m}, \quad \mathrm{id}: T \rightarrow T.
$$
Indeed, ${\iota}_a,{\iota}_b$ and $\mathrm{id}$ all extend to $T*\ZZ_m$, while the universal property of the free product yields ${\phi}$.

Since the generator $[1]_m$ of $\ZZ_m$ and all of $T$ are in the image of ${\phi}$, the latter is surjective. Also, the restrictions of ${\phi}$ to each of the subgroups $T$, $\ZZ_a$, and $\ZZ_b$ are injective since they correspond exactly to the post-compositions of $\mathrm{id}$, ${\iota}_a$, and ${\iota}_b$, respectively, with their inclusions in $T*\ZZ_m$.

Now, fix a finite group ${\Gamma}$ to be realised as a quotient ``normaliser/subgroup'' of $T*\ZZ_a*\ZZ_b$. We know, by the hypothesis, that there exists a free subgroup $H{\leq}T*\ZZ_m$, of finite index, such that $N(H)/H {\cong} {\Gamma}$. Since ${\phi}$ is surjective, ${\phi}^{-1}(H)$ is the desired subgroup by Lemma~\ref{lemma:surjection-telescopic}, once we verify that it is free. To this end, assume that $\phi^{-1}(H)$ is not. Then, by Kurosh's theorem, $\phi^{-1}(H)$ contains a conjugate of a non-free subgroup of $T$, $\ZZ_a$ or $\ZZ_b$. In the first case, there exists $R \leq T$ non-free such that $\phi^{-1}(H)$ contains $wRw^{-1}$, for some $w$, as a free factor. Then
$$
H \geq \phi(wRw^{-1}) = \phi(w)\phi(R)\phi(w^{-1}) = \phi(w)R\phi(w^{-1}),
$$
where the third equality stems from the fact that $\phi$ restricts to ``the identity + inclusion'' on $T$, by definition.  This implies that $H$ contains a non-free subgroup, which is a contradiction. 
    
Similarly, assume that ${\phi}^{-1}(H)$ contains a free factor of the form $wCw^{-1}$, with $C$ a non-free subgroup of $\ZZ_a$.
Then we have
$$
H \geq \phi(wCw^{-1}) = \phi(w)\phi(C)\phi(w^{-1}) {\cong} \phi(w){\iota}_a(C)\phi(w^{-1}),
$$
where, once again, the third equality stems from the fact that ${\phi}$, when restricted to $\ZZ_a$, is just ${\iota}_a$, plus the inclusion of $\ZZ_m$ in $T*\ZZ_m$. Since ${\iota}_a$ is injective, and $C$ is not free, then ${\iota}_a(C)$ is not free either, so that $H$ contains a non-free subgroup, which is again a contradiction. An analogous reasoning applies if we assume that $C$ is a non-free subgroup of $\ZZ_b$, and the lemma follows.
\end{proof}

\begin{lemm}\label{lemma:telescopic-Z}
Assume that a product of the form $T*\ZZ_m$, $m\geq 2$, is freely telescopic. Then $T*\ZZ$ is freely telescopic.
\end{lemm}
\begin{proof}
As above, consider the morphisms $\mathrm{id}: T{\to}T$ and $q: \ZZ{\to}\ZZ_m$, together with the induced surjective ``composite'' morphism
\[
{\phi}: T*\ZZ {\to} T*\ZZ_m. 
\]
In this case, the injectivity of ${\phi}$ on the $\ZZ$-factor is not required: we already know that any subgroup of $\ZZ$ is free.
\end{proof}

\begin{thm}\label{thm:free-products-are-telescopic}
Any free product of at least two non-trivial cyclic groups is freely telescopic, except for $\ZZ_2*\ZZ_2$.
\end{thm}

\begin{proof}
If $T = \ZZ_2*\ZZ_2$, the infinite dihedral group, then the only Schreier graphs associated with finite-index free subgroups of $T$ are cycles, and $T$ cannot be freely telescopic (obviously, it cannot be telescopic at all).  If $T = \ast^{n}_{i=1} \ZZ_{p_i}$ is a finite free product of non-trivial cyclic groups with $n \geq 2$ and $p_1 \geq 3$ or $n \geq 3$, the proof proceeds by induction on $n$.  

If $n=2$, then Proposition~\ref{prop:ZpZq-telescopic} yields that $T = \ZZ_{p_{1}}*\ZZ_{p_{2}}$ is freely telescopic, assuming that $p_{1}  \geq 3$ and $p_{2} \geq 2$, without loss of generality. If $n=3$, either each of $p_{1},p_{2},p_{3}$ equals $2$, in which case Proposition~\ref{prop:Z2Z2Z2-telescopic} yields free telescopicity of  $\ZZ_{2}*\ZZ_{2}*\ZZ_{2}$, or, without loss of generality, we have $p_{3}\neq2$. In the latter case, $\ZZ_{p_{1}}*\ZZ_{\lcm(p_{2},p_{3})}$ is already freely telescopic, and Lemma~\ref{lemma:telescopic-lcm} shows that  $\ZZ_{p_{1}}*\ZZ_{p_{2}}*\ZZ_{p_{3}}$ is so, as well.
    
The inductive step towards $n \geq 4$ is made by using Lemma~\ref{lemma:telescopic-lcm}, with $p_i$'s being finite. Setting $p_i=\infty$, for any $i$, also yields freely telescopic groups, by Lemma~\ref{lemma:telescopic-Z}. This concludes the proof of the theorem.
\end{proof}

\section{Graph substitution}\label{graph-substitution}
In this section we shall always consider directed labelled graphs. Our goal is to define a reasonable condition  that allows replacing vertices and edges of a Cayley graph by other ``chunks of graphs'' in a way that preserves the automorphism group. All edges having a fixed given label will be replaced by the same ``chunk'' of a suitable graph, for each label, and a similar procedure takes place for all vertices. In what follows, let $\Gamma$ be a finite group generated by a finite set $S$, and let $\cC := \Cay(\Gamma, S)$ be its Cayley graph.

\subsection{Vertex graphs}\label{vertex-graphs-substitution}

If $v$ is a vertex in the Cayley graph $\cC = \Cay(\Gamma, S)$, let $\adj\, v$ be the set of tuples consisting of edges adjacent to $v$ and their orientations relative to $v$, i.e. $\adj\, v = (\St_+ v){\times}\{+\} {\sqcup} (\St_- v){\times}\{-\}$. Let ${\Sigma}:=S{\times}\{+,-\}$, then each set $\adj\, v$ is naturally in bijection with ${\Sigma}$: one can easily identify the corresponding pairs of edge labels and their orientations relative to $v$. Let ${\tau}_v: \adj v {\to} {\Sigma}$ be this bijection, to which we shall refer as \textit{the signature} of $v$.

If $\cV$ is a connected graph and ${\chi}:{\Sigma}{\hookrightarrow}V\cV$ is an injective map, then we call the pair $(\cV,{\chi})$ a \textit{signed graph}. Looking ahead, the map ${\chi}$ will tell us how one should connect the edges adjacent to a vertex $v \in V\cC$ to the vertices of $\cV$ when replacing $v$ in $\cC$ by $\cV$. The injectivity of ${\chi}$ also comes useful later.
         
Let ${\partial}\cV:= \im {\chi}$, which we shall call the boundary of $\cV$. Fix a signed graph $(\cV, \chi)$, and let
\[
\aut(\cV,{\chi}) := \{{\phi} \in \aut \cV\ :\ {\phi}|_{{\partial}\cV} = \mathrm{id}_{{\partial}\cV}\}
\]
be the group of automorphisms of $\cV$ which restrict to the identity map on the boundary ${\partial}\cV$. We shall call $\aut(\cV, {\chi})$ the group of \textit{signed automorphisms} of $(\cV, {\chi})$.
        
Now let us consider the following condition, that will play an important role in Section~\ref{section:main-statement-substitution}. 
\begin{enumerate}
\renewcommand{\labelenumi}{\textbf{S.\theenumi}}
\item \label{S1}
$\aut (\cV, {\chi}) = \{\mathrm{id}\}$. 
\end{enumerate}
In other words, S.\ref{S1} states that any non-trivial automorphism of $\cV$ has to move some vertex of its boundary ${\partial}\cV$.

For a given vertex $v {\in} V\cC$, a signed graph $\cV$ can be inserted in place of $v$ by connecting each edge $e\, {\in}\, \adj\, v$ to the vertex ${\chi}({\tau}_v(e))\,{\in}\,{\partial}\cV$. Broadly speaking, condition S.\ref{S1} forbids any automorphism local to $\cV$ to appear when $v$ is replaced by $\cV$.

A graph $\cV$ described above will be called the \textit{vertex-graph} associated with $v$.

\subsection{Edge graphs}\label{edge-graphs-substitution}

The case of edge substitution is simpler: if $s$ is a label, an \textit{edge-graph} for $s$ is a connected graph $\cE_s$ with distinct distinguished vertices $h^{+}_s$ and $h^{-}_s$. We substitute an edge $e$ labelled $s$ by first removing $e$, and then identifying the origin of $e$ with $h^{+}_s$ and its terminus with $h^{-}_s$.

\subsection{Vertex and edge substitution}\label{sssection:substitution}
        
Below we describe the complete substitution procedure.  Recall that we start by considering the Cayley graph $\cC$ of a group $\Gamma$ with respect to generators $S$, as well as 
\begin{itemize}
\item a connected signed vertex graph $(\cV,{\chi})$, 
\item a connected edge graph $\cE_s$, for each edge label $s$, with two distinct distinguished vertices $h_s^{+}$ and $h_s^{-}$ and trivial automorphism group $\aut \cE_s = \{\mathrm{id}\}$.
\end{itemize}
        
Let $\cC'$ be the result of replacing each vertex $v$ of $\cC$ with an instance of $\cV$ as explained in Section~\ref{vertex-graphs-substitution}. More precisely, if $v$ is a vertex with signature ${\tau}_v:\adj v {\to} {\Sigma}$, we remove $v$, insert a copy of $\cV$ and connect each edge $e{\in}\adj v$ to the vertex ${\chi}({\tau}_v(e))$ of $\cV$.

Now, let $\cC^*$ be the result of replacing each ``old'' edge (i.e. an edge that is not in any of the vertex graphs) labelled $s$ in $\cC'$ by a copy of $\cE_s$, with $h_s^{+}$ identified with the origin of $e$, and $h_s^{-}$ with its terminus, as described in Section~\ref{edge-graphs-substitution}.

Here and below, copies of $\cV$ will be always called ``vertex-graphs'', and copies of $\cE_s$ will be ``edge-graphs''. 

For $v{\in}V\cC$, write $\cV_v$ for the instance of the vertex-graph $\cV$ inserted in place of $v$, and let ${\iota}_v:\cV {\hookrightarrow} \cC^*$ be the corresponding graph embedding. Similarly, for $e{\in}E\cC$ with label $\mu(e)=s$, we shall write $\cE_e$ for the instance of the edge-graph labelled $s$ which is inserted in place of $e$, and ${\iota}_e:\cE_{s}{\hookrightarrow} \cC^*$ will be the corresponding graph embedding.

If $e$ is an edge with origin $u$, resp. terminus $v$, and label $s$, then we identify ${\iota}_e(h_s^{+})$ with $ {\iota}_u({\chi}(s, +))$, resp. ${\iota}_e(h_s^{-})$ with ${\iota}_v({\chi}(s, -))$. The vertex that we obtain after such identification is shared between $\cE_e$ and $\cV_u$, resp. $\cV_v$. We shall call these vertices $h^{+}(e)$ and $h^{-}(e)$, respectively, in order to distinguish them in $\cC^*$.

Observe that our construction implies the following:
\begin{itemize}
\item The vertex-graphs are disjoint from each other, since $h_s^+$ and $h_s^-$ are distinct in any $\cE_s$, and so are the edge-graphs, by injectivity of ${\chi}$.
\item The vertices of ${\partial}\cV_v$ are exactly those of the form $h^{\pm}(e)$ for an edge $e$ adjacent to $v$ in the initial graph $\cC$.
\item $h^{+}(e) {\in} \cV_u$ if and only if $u$ is the origin of $e$, and $h^{-}(e) {\in} \cV_v$ if and only if $v$ is the terminus of $e$.
\item ${\chi}^{-1}{\iota}_u^{-1}(h^+(e)) = ({\mu}(e), +)$, resp. ${\chi}^{-1}{\iota}_u^{-1}(h^-(e)) = ({\mu}(e), -)$, if $u$ is the origin, resp. the terminus, of $e$.\end{itemize}

Finally, consider an ordering $s_{1} < \dots < s_n$ on the edge labels. Let, for each $i=0, \dots, n$,  $\cC^*_{(i)}$ be the subgraph of $\cC^*$ that consists of the vertex-graphs and only the  edge-graphs corresponding to the labels $s_j$ with $j \leq i$. In other words, this is a subgraph obtained by removing the ``interiors'' (i.e. everything but the vertices $h^{\pm}(e)$) of all the edge-graphs corresponding to the labels $s_j$ for $j > i$.

\subsection{Main statement}\label{section:main-statement-substitution}

Below we formulate the main statement regarding our graph substitution procedure. 

\begin{prop} \label{prop:graph-substitution}
With the notation above, if \textbf{S.\ref{S1}} holds, as well as if
\begin{enumerate}[resume]
\renewcommand{\labelenumi}{\textbf{S.\theenumi}}
\item \label{S2}
one can order the labels $s_{1}, \dots, s_n$ in such a way that for any $i\in\{0, \dots, n-1\}$, the graph $\cC^*_{(i)}$ contains no subgraph isomorphic to $\cE_{s_{i}}$ except for the edge-subgraphs $\cE_{e}$, for $e\in E\cC$ with label ${\mu}(e)=s_{i}$,
\end{enumerate}
then
\[
\aut \cC^* {\cong} \aut \cC.
\]
\end{prop}

Here, condition S.\ref{S2} ensures that under any automorphism of $\cC^*$ the instances of $\cV$, resp. the instances of any $\cE_s$, must be sent to each other. As can be understood from the proof, any other condition ensuring this fact can be used instead of S.\ref{S2}. Then, it is enough to define, given an automorphism of $\cC^*$, a corresponding automorphism of $\cC$ by looking at the correspondence 
$$\{\text{vertex / edge of } \cC\} {\leftrightarrow} \{\text{vertex-graph / edge-graph in } \cC^*\}.$$ 
Condition S.\ref{S1} ensures that not too much liberty is gained by making the aforementioned substitutions.

\begin{proof}[Proof of Proposition~\ref{prop:graph-substitution}] For $v{\in}V\cC$, recall that $\cV_v$ stands for the instance of $\cV$ inserted in place of $v$, and ${\iota}_v:\cV {\hookrightarrow} \cC^*$ is the corresponding graph embedding. For $e{\in}E\cC$,  $\cE_e$ stands for the instance of the edge-graph with label $\mu(e)$, inserted in place of $e$, while ${\iota}_e:\cE_{{\mu}(e)}{\hookrightarrow} \cC^*$ is the corresponding embedding.

\medskip
            
\paragraph{\textbf{Constructing a morphism} $\Con: \aut \cC^* {\to} \aut \cC$} Fix an automorphism $\phi \in \aut \cC^*$. We claim that, for any edge $e{\in}E\cC$, there exists a unique $e'{\in}E\cC$ such that ${\phi}$ restricts to an isomorphism ${\phi}|_{\cE_e}:\cE_e {\to} \cE_{e'}$ and ${\mu}(e) = {\mu}(e')$.  From this follows that ${\phi}:h^{+}(e) \mapsto h^{+}(e')$ and ${\phi}:h^{-}(e) \mapsto h^{-}(e')$, since $\cE_{{\mu}(e)}$ has trivial automorphism group.

Let, without loss of generality, $s_{1} < \dots < s_n$ be the ordering on the edge labels required by condition S.\ref{S2}. We shall verify by reverse induction on $1{\leq}i{\leq}n$ that 

\begin{enumerate}[resume]
\renewcommand{\labelenumi}{\textbf{($\ast$)}}
\item if $e$ has label $\mu(e) = s_i$, then there exists a unique $e'$ such that ${\phi}|_{\cE_e}$ is an isomorphism from $\cE_e$ to $\cE_{e'}$, and $\mu(e')=s_i$ (as noted above, this implies that we have ${\phi}|_{\cE_e}: h^{\pm}(e) \mapsto h^{\pm}(e')$).
\end{enumerate}           

If $i=n$, we know by the hypothesis (since condition S.\ref{S2} is satisfied) that the only subgraphs of $\cC^*=\cC^*_{(n)}$ isomorphic to $\cE_{s_{n}}$ are the graphs $\cE_e$ with $e$ an edge labelled $s_{n}$. Fix such an edge: then ${\phi}|_{\cE_e}$ defines an isomorphism onto its image, which must therefore be of the form $\cE_{e'}$, for some $e'$ of label $s_{n}$. This provides the induction base.
            
Now, fix $i \leq n-1$ and assume that property $(\ast)$ holds for any $i+1{\leq}j{\leq}n$. Since, for all $i+1{\leq}j{\leq}n$, ${\phi}$ sends instances of $\cE_{s_j}$ to instances of $\cE_{s_j}$, and their vertices $h^{\pm}(e)$ to themselves, it restricts to an automorphism of $\cC^*_{(i)}$. Then, by applying S.\ref{S2}, we know that ${\phi}$ must send an instance of $\cE_{s_{i}}$ to another one. Indeed, ${\phi}|_{\cE_{s_{i}}}$ has range in $\cC^*_{(i)}$, while S.\ref{S2} guarantees that its image must then be an instance of $\cE_{s_{i}}$. This proves the induction step and our claim is thus verified.

Now, let us show that for each vertex $v\in V\cC$, there exists a unique $v'\in V\cC$ such that ${\phi}$ restricts to an isomorphism ${\phi}|_{\cV_v}:\cV_v {\to} \cV_{v'}$. Indeed, it follows from the above that ${\phi}$ sends edge-graphs to edge-graphs, and boundary vertices (those of the form $h^{\pm}(e)$) to boundary vertices, hence it restricts to an automorphism of the subgraph of $\cC^*$ consisting only of the vertex-graphs employed in the construction. The latter is just a disjoint union of all vertex-graphs. Since ${\phi}(\cV_v)$ is connected, it must lie in some $\cV_{v'}$, and thus coincide with it.
            
Let $ {\psi} :=\Con({\phi})$ be defined as follows: ${\psi}(v)$ is the unique $v'$ such that ${\phi}$ sends $\cV_v$ to $\cV_{v'}$, and ${\psi}(e)$ is the unique $e'$ such that ${\phi}$ sends $\cE_e$ to $\cE_{e'}$.

Observe that ${\psi}$ is bijective since applying the above construction to ${\phi}^{-1}$ yields an inverse map to ${\psi}$, and it remains to verify that ${\psi}$ is actually a morphism of graphs, i.e. $\psi$ preserves adjacency.

Recall that $h^{+}(e){\in}\cV_u$ if and only if $u$ is the origin of $e$. Fix an edge $e{\in}E\cC$ with origin $u{\in}V\cC$, such that $h^{+}(e){\in}\cV_u$. Let $u':={\psi}(u)$ and $e':={\psi}(e)$. Then
\[
h^{+}(e') = {\phi}(h^{+}(e)) {\in} {\phi}(\cV_u) =  \cV_{u'},
\]
and since $h^{+}(e'){\in}\cV_{u'}$, $u'$ is the origin of $e'$. Replacing $+$ by $-$ and ``origin'' by ``terminus'', we conclude that ${\psi}$ is indeed a morphism of graphs.

\medskip

\paragraph{$\Con$ \textbf{is a homomorphism}} Let ${\phi}, {\phi}'{\in}\aut(\cC^*)$ be two automorphisms, and suppose that ${\phi}$ sends $\cV_v$ to $\cV_{v'}$, while ${\phi}'$ sends $\cV_{v'}$ to $\cV_{v''}$. Then ${\phi}'{\phi}$ sends $\cV_{v}$ to $\cV_{v''}$, and $\Con({\phi}'{\phi})(v) = v'' = \Con({\phi}')(\Con({\phi})(v))$.
An analogous statement holds for edges, and thus $\Con$ is a morphism of groups.

\medskip

\paragraph{\textbf{Injectivity of} $\Con$} Let $\phi$ be an element of $\aut \cC^{*}$. If $e, e' \in E \cC$ are two edges, then $\phi(\cE_e)=\cE_{e'}$ if and only if $\Con (\phi )(e)=e'$. If $\Con(\phi) = \mathrm{id}$, then we have that $\phi(\cE_e)=\cE_{e}$. However, edge graphs have trivial automorphism group, and thus $\phi|_{\cE_e}=\mathrm{id}$. Since the graph $\cC^*$ is folded, we obtain that $\phi$ is the identity on $\cC^*$.

\medskip

\paragraph{\textbf{Surjectivity of} $\Con$} Let ${\psi}{\in}\aut \cC$. Our goal is to find some ${\phi}{\in}\aut\cC^*$ such that $\Con({\phi}) = {\psi}$.
Define
\begin{alignat*}{10}
{\phi}|_{\cE_e} &:= {\iota}_{{\psi}(e)}{\circ}({\iota}_e|_{\cE_e})^{-1}, \qquad &\text{ for every edge } e {\in} E\cC,\\
{\phi}|_{\cV_v} &:= {\iota}_{{\psi}(v)}{\circ}({\iota}_v|_{\cV_v})^{-1}, \qquad &\text{ for every vertex } v {\in} V\cC.
\end{alignat*}
            
By construction, we see already that, assuming ${\phi}$ to be a well-defined automorphism, its image under $\Con$ is ${\psi}$, and it remains to verify that the piecewise definitions above actually agree. The only vertices lying in two different subgraphs (one vertex- and one edge-subgraph) are the boundary vertices. Fix an edge $e$ with terminus $v$ and label $s$, such that ${\iota}_e(h_s^{-}) = h^{-}(e) = {\iota}_v({\chi}(e, -))$. Then
\begin{alignat*}{10}
{\phi}|_{\cE_e}(h^{-}(e)) &= {\iota}_{{\psi}(e)}{\iota}_e^{-1}{\iota}_e(h_s^{-}) = {\iota}_{{\psi}(e)}(h_s^{-}) = h^{-}({\psi}(e)),\\
{\phi}|_{\cV_v}(h^{-}(e)) &= {\iota}_{{\psi}(v)}{\iota}_v^{-1}{\iota}_v({\chi}(e, -)) = {\iota}_{{\psi}(v)}({\chi}(e, -)) = h^{-}({\psi}(e))
\end{alignat*}
and ${\phi}|_{\cE_e}(h^{-}(e)) = {\phi}|_{\cV_v}(h^{-}(e))$, as desired. Replacing $-$ by $+$ and ``terminus'' by ``origin'', we conclude that ${\phi}|_{\cE_e}(h^{+}(e)) = {\phi}|_{\cV_v}(h^{+}(e))$, as well.
            
Finally, observe that ${\phi}$ is an automorphism since the construction above applied to ${\psi}^{-1}$, the inverse of $\psi$, yields an inverse of ${\phi}$.\end{proof}

\section{Adding more factors}

In this section we shall concentrate mostly on properties of telescopic groups, which do not have direct applications to the combinatorial results from Section~\ref{section:symmetries}, and rather stay on the purely group-theoretic side of our study.

\begin{prop}
Let $G_1$ be a freely telescopic group and let $G_2$ be a group having a finite index subgroup $H$ such that $N_{G_2}(H) = H$. Then $G_1*G_2$ is freely telescopic.
\end{prop}
    
Note that, in particular, the free product of two freely telescopic groups is freely telescopic. The proof follows the same kind of argument as Lemmas~\ref{lemma:telescopic-lcm} -- \ref{lemma:telescopic-Z}.
    
\begin{proof} Let $\Gamma$ be a finite group. Since $G_1$ is freely telescopic, there exists a finite-index free subgroup $H_1$ satisfying $\sfrac{N_{G_1}(H_1)}{H_1} \cong \Gamma$. By assumption, there exists a finite-index free subgroup $H_2 < G_2$ such that $\sfrac{N_{G_2}(H_2)}{H_2}$ is trivial.

Consider the inclusions morphisms ${\phi}_i:G_i {\to} G_1{\times}G_2$, for $i = 1, 2$, and the morphism ${\phi}:G_{1}*G_{2}{\to} G_{1}{\times}G_{2}$ induced by the universal mapping property of the free product. The latter is easily seen to be surjective. 

Observe that $H_1 \times H_2$ is a finite index subgroup of $G_1 \times G_2$, and that the normaliser $N_{G_{1}{\times}G_{2}}(H_{1}{\times}H_{2})/H_{1}{\times}H_{2}$ is isomorphic to $N_{G_{1}}(H_{1})/H_{1} {\times} N_{G_{2}}(H_{2})/H_{2}$, while the latter is simply $N_{G_{1}}(H_{1})/H_{1}$. Then, let us consider the preimage $\widetilde{H}$ of $H_1 \times H_2$ under the morphism $\phi$ defined above. Since $\phi$ is an epimorphism, by Lemma~\ref{lemma:surjection-telescopic}, we have immediately that $G_1 \ast G_2$ is telescopic. It remains to show that $\widetilde{H}= \phi ^{-1} ( H_1 \times H_2)$ is a free subgroup of $G_1 \ast G_2$.

By Kurosh's Subgroup Theorem, $\widetilde{H}$ can be written as $ \widetilde{H} = F(X) * (*_i u_iK_{1,i}u_i^{-1}) * (*_j w_jK_{2,j}w_j^{-1})$ for a number of subgroups $K_{1,i}$ of $G_1$, $K_{2,j}$ of $G_2$, a subset $X$ of $G_1 \ast G_2$ and words $u_i , w_j$ in $G_1 \ast G_2$. If $\widetilde{H}$ were not free, then some of the $K_{1,i} , K_{2,j}$ would not be free either. Assume, without loss of generality, that $K=K_{1,i}$ is not free. Then, by construction, $u_1 K u_1^{-1}$ is a non-free subgroup of $\widetilde{H}$ and thus
\[
\phi(u_1 K u_1^{-1}) = {\phi}(u_{1}){\phi}(K){\phi}(u_{1})^{-1} = {\phi}(u_{1}){\phi}_{1}(K){\phi}(u_{1})^{-1}
\]
is a non-free subgroup of $H_1$, by using the injectivity of ${\phi}_{1}$. Therefore $H_{1}$ cannot be free, and the proposition follows.
\end{proof}

\section{An asymptotic estimate}\label{asymptotic-estimate}

\begin{thm}\label{thm:asympt}
Let $T$ be a finite free product of cyclic groups, different from $\ZZ_2*\ZZ_2$. Then for any finite group ${\Gamma}$, there exist constants $A > 1$, $B > 0$ and $M\in\NN$ such that for all $d{\geq}M$ the set $F(T, {\Gamma}, d) = \{\text{free subgroups } H < T \text{ of index } \leq d \text{ with } N_T(H)/H{\cong}{\Gamma}, \text{ up to conjugacy}\}$ has cardinality ${\geq}A^{B d\log d}$.
\end{thm}

The proof of the above theorem is conceptually simple.  Fix the aforementioned groups $T$, $\Gamma$ and index $d$, and let $N = N(T, \Gamma, d)$ be an integer whose dependence on $T, \Gamma, d$ will be clarified later. We shall verify that sufficiently many non-isomorphic graphs $\cH_{\sigma}$ on $N$ vertices  can be built, with respect to an additional parameter $\sigma$, introduced below.  Then the previously used edge-links in the construction of the Schreier graph for $H < T$ (based on a Cayley graph of $\Gamma$) will be combined with one of many possible choices of $\cH_{\sigma}$. Once we show that there are $\geq A^{B d \log d}$ non-isomorphic instances of $\cH_\sigma$, the result follows. On the other hand, the cardinality of $F(T, \Gamma, d)$ is $\leq C^{D d\log d}$, for some constants $C > 1$, $D > 0$, with $d \geq M$ by \cite{Bollobas}. Thus, the \textit{growth type} of the cardinality of $F(T, \Gamma, d)$ is $d^d$, and is independent of $T$ and $\Gamma$.

One natural condition on the graphs $\cH_{\sigma}$ is that they do not contain ``roots'' that have property $P$ from Section~\ref{section:constructionofthelinks}, i.e. vertices $v$ for which 
\begin{itemize}
\item $v \cdot r = v  \cdot c$, if $T = \ZZ_p*\ZZ_q$ (with $p\geq q$, $p\geq 3$, $q\geq 2$), or
\item $v \cdot rgbgrgr$ = $v$ and $v \cdot b = v \cdot g$, if $T = \ZZ_2*\ZZ_2*\ZZ_2$.
\end{itemize}
If an $\cH_\sigma$ contained such a root, then the combination of an edge-link with $\cH_\sigma$ would not have a \emph{unique} root any more.

The following lemma provides all the necessary details about constructing $\cH_\sigma$'s, which is the first step towards the proof of Theorem \ref{thm:asympt}.

\begin{lemm}\label{lemma:manygraphs}
Let $T$ be $\ZZ_p*\ZZ_q$, with $p\geq q$, $p\geq 3$, $q\geq 2$, resp. $\ZZ_2*\ZZ_2*\ZZ_2$. Then there exist constants $k, K > 0$,  such that for any $N$ a multiple of $pq$, resp. $N$ a multiple of $8$, we have at least $KN^{kN}$ non-isomorphic non-degenerate $(T, \{r, c\})$-graphs, resp. $(T, \{r, g, b\})$-graphs, with two split vertices.
\end{lemm}

\begin{proof}
First, we treat the case $T = \ZZ_p*\ZZ_q$, with $p\geq q$, $p\geq 3$ and $q \geq 2$. Let $N$ be a multiple of $pq$, and let us consider $N$ vertices ordered as $v_{0}, \dots, v_{N-1}$. Similar to Section~\ref{section:constructionofthelinks}, we start by drawing red $p$-cycles of the form
\[
v_{kp} \rto v_{kp+1} \rto \dots \rto v_{kp+p-1} \rto v_{kp},
\]
for each $k=0, \dots, N/p-1$. Then we draw a cyan edge $v_0 \cto v_1$, which creates a ``double edge'' from $v_{0}$ to $v_{1}$, and continue by adding as few cyan $q$-cycles as possible in order to obtain a connected graph, without creating any more double edges. Let $\cH_\emptyset$ denote the resulting graph. 

One sees that since the initial graph is a disjoint union of red $p$-cycles, then $\leq 2N/p$ vertices will have to be joined by cyan edges in order to get a connected graph, and then at most $q$ other vertices will be joined in order to close up the last cyan $q$-cycle. Observe that at least $N- 2N/p - q \geq N/3 - q \geq N/4$ (for $N$ large enough) vertices will remain ``free'' in $\cH_\emptyset$. Let $F \geq N/4$ be the exact number of remaining free vertices, and let $D:= \lfloor F \rfloor_q$, where $\lfloor x\rfloor_q$ denotes the greatest multiple of $q$ that is smaller than $x$, for any natural numbers $q$ and $x$. Then, for $N$ large enough, we also have $D \geq N/4$.

Now, let us consider any ordering $\sigma$ on the free vertices in $\cH_{\emptyset}$. By drawing cyan edges between the consecutive vertices in $\sigma$, and closing up the cycles when necessary, any such choice of $\sigma$ yields a non-degenerate graph $\cH_{\sigma}$.
   
First of all, observe that after fixing a basepoint $v_0$ in $\cH_\emptyset$, the graphs $\cH_{\sigma}$ and $\cH_{\tau}$, for any orderings $\sigma$ and  $\tau$, are basepoint-isomorphic if and only if $\cH_{\sigma} = \cH_{\tau}$.  Indeed, since $\cH_\emptyset \leq \cH_{\sigma}$, any isomorphism $\cH_{\sigma} \to \cH_{\tau}$ restricts to an embedding $\cH_\emptyset \hookrightarrow \cH_{\tau}$, only one of which exists by foldedness: namely, the identity map.

Our next step is estimating the number of non-isomorphic graphs $\cH_\sigma$, as above, without a basepoint. To this end, observe that two orderings $\sigma$ and $\tau$ of the remaining free vertices yield the same graph if and only if $\tau$ can be obtained from $\sigma$ by a permutation of the respective ``$q$-blocks'' (each consisting of the vertices in a cyan cycle created in accordance with $\sigma$ or $\tau$), as well as by cyclic permutations within each ``block''. Indeed, permuting the blocks, together with cyclic permutations inside each block, obviously does not change the resulting graph. Also, if two orderings result in the same graph, then the graph isomorphism implies that the orderings are equal up to the aforementioned permutations. Thus, the number of non-basepoint-isomorphic extensions of $\cH_\emptyset$ by using orderings is \[
\geq \frac{D!}{(D/q)!\, q^{D/q}}.
\]

However, since the resulting graphs will be used later on in the constructions of ``edge-links", we do not want them to contain any double edges. This mean that no vertices $v$ with property $P: v \cdot r=v \cdot c$, except for $v_0\cto v_1$ (joined by a double edge by construction) are allowed, since this would contradict the uniqueness of roots (in the sense of having property $P$ from Section~\ref{section:constructionofthelinks}).

We will therefore consider only the orderings that will not produce such edges. Therefore, we start with $D$ free vertices and at each step choose any of the remaining vertices, with the restriction that no double edge be created. Observe that there are three possible cases: the to-be-chosen vertex can be the first of a new cyan cycle, it can lie in the middle of such a cycle or be its last one with respect to the chosen ordering. In the first case, no restriction is imposed since choosing the first vertex of a cycle does not create edges at all. In the second case, one edge is created, and thus one vertex is ``illegal'' in the sense that choosing it would result in a double edge. In the third case, two vertices are ``illegal''.

We have to determine if any initial choice of vertices in cyan cycles can be extended to a legal ordering as above. Such an extension may fail under certain circumstances: e.g. if two vertices remain to be ordered, then no legal choice may be possible. In order to circumvent this problem, we shall only consider orderings of $D-q$ vertices, since any legal prefix of length $D-q$ \emph{can} be extended to a legal ordering of length $D$ (if $q=2$,  then one has to consider orderings of the first $D-4$ vertices instead).

With the above points in mind, one readily sees that the number of legal orderings is at least ${\geq} (D-2)(D-3) \dots (q-1) = (D-2)!/(q-2)!$ Once again, taking into account the fact that orderings yield non-basepoint isomorphic graphs if and only if they are obtained one from another by permutations of the blocks and by cyclic permutations within those blocks, we conclude the existence of
\[
\geq \,\, \frac{1}{(q-2)!} \cdot \frac{(D-2)!}{(D/q)!\, q^{D/q}} \tag{$\ast$}
\]
legal choices  of $\sigma$ defining non-isomorphic graphs $\cH_\sigma$ without basepoint.

Now, take any $\cH_{\sigma}$, split its vertex $v_{0}$, and call the resulting graph $\cH_{\sigma}'$, while $v_r, v_c$ will be its dangling vertices.  We claim that:
\begin{enumerate}
\item each $\cH_{\sigma}'$ is connected; 
\item if $\sigma$ and $\tau$ are distinct (up to the aforementioned permutations), then $\cH_{\sigma}'$ and $\cH_{\tau}'$ are non-isomorphic without basepoint;
\item $\cH_{\sigma}'$ does not have double edges.
\end{enumerate}

The first property holds since we split at the \textit{unique} double edge of $\cH_\sigma$. The second property follows from the fact any isomorphism between $\cH_{\sigma}'$ and $\cH_{\tau}'$ must fix the dangling vertices. Then, were the split graphs $\cH_{\sigma}'$ and $\cH_{\tau}'$ isomorphic, then $\cH_{\sigma}$ and $\cH_{\tau}$ would be so. Finally, since we always split the vertex $v_0$ of $\cH_\sigma$ at the end of its unique double edge, none of the latter remains, and the third property is verified.
    
Thus we conclude the existence of as many as ($\ast$) non-isomorphic graphs $\cH_{\sigma}'$, without a root and having two dangling vertices. By invoking Stirling's formula and using the fact that $D-2 \geq \frac{D}{q}$ for $D$ big enough, the desired asymptotic estimate follows.

The argument for $T = \ZZ_2*\ZZ_2*\ZZ_2$ is similar and not spelled out in detail here. First, we take $N$ a big enough multiple of $8$ and then assemble a chain of alternating red-green edges connecting $N$ vertices. Finally, we estimate the number of different graphs obtained by adding blue edges in an analogous way to the above argument.
\end{proof}

\begin{proof}[{Proof of Theorem~\ref{thm:asympt}}] 
Observe that it suffices to verify the statement for $T = \ZZ_p*\ZZ_q$, with $p\geq q$, $p \geq 3$, $q\geq 2$, and $T = \ZZ_2*\ZZ_2*\ZZ_2$, only. Indeed, for any surjection $f:G{\to}K$  and any subgroups $K_{1}, K_{2}  \leq K$, their pre-images $f^{-1}(K_{1})$ and $f^{-1}(K_{2})$ are conjugate if and only if $K_{1}$ and $K_{2}$ are so. Therefore, the inductive constructions of Section~\ref{section:addingfactors} preserve the size of the set in the statement of the theorem. 
    
Assume that $T$ is either $\ZZ_p*\ZZ_q$, as above, or $\ZZ_2*\ZZ_2*\ZZ_2$ and choose $N$ as in Lemma~\ref{lemma:manygraphs}. Then there are at least $KN^{kN}$ non-isomorphic graphs $\cH'_{\sigma}$ with two dangling vertices.

Let us fix such a graph $\cH_{\sigma}$. Then consider the edge-link $\cL_e$ constructed in Section~\ref{section:constructionofthelinks}, and glue $\cH'_{\sigma}$ to $\cL_e$ via the dangling vertices in order to obtain $\cL_e+\cH'_{\sigma}$. The graph $\cL_e+\cH'_{\sigma}$ still has a unique root, which is the only condition for the arguments of Section~\ref{section:basiccases} to hold. Following those constructions once again, for a given finite group ${\Gamma}$, results in a finite-index subgroup $H_{\sigma} < T$ with $N_T(H_{\sigma})/H_{\sigma}{\cong}{\Gamma}$.

For any two choices of graphs $\cH'_{\sigma}$ and $\cH'_{\tau}$ as above, the resulting subgroups $H_{\sigma}$ and $H_{\tau}$ are not conjugate. Indeed if $H_{\sigma}$ and $H_{\tau}$ were conjugate, then $\Sch_{T,X}(H_{\sigma})$ and $\Sch_{T,X}(H_{\tau})$ would be isomorphic (with base-points not necessarily matched), so that the roots of the former Schreier graph would be sent to the roots of the latter one, and the instances of $\cH'_{\sigma}$ would be mapped isomorphically to the instances of $\cH'_{\tau}$, which would imply that $\sigma$ and $\tau$ coincide up to the above mentioned permutations. Thus, there are ${\geq}KN^{kN}$ non-conjugate free subgroups $H_\sigma$ of $T$ with $N_T(H_\sigma)/H_\sigma {\cong} {\Gamma}$ having same index.

Finally, let us estimate the index of $H$.  For any fixed ${\Gamma}$ with a given generating set $S$, let $e_\Gamma$ and $v_\Gamma$ be the following numbers:
\begin{alignat*}{10}
e_{\Gamma} :=& \text{the number of edge-links appearing in the original construction}\\
                     & \text{of Section~\ref{section:constructionofthelinks}} =  |\Gamma| \cdot |S| \cdot (|S|-1)/2;\\ 
v_{\Gamma} :=& \text{ the number of vertices appearing in the original construction.}
\end{alignat*}
Now, with the modified construction involving $\cH'_{\sigma}$, we obtain that the index of the corresponding subgroup $H_\sigma$ is
\[
d = v_{\Gamma} + e_{\Gamma} N,
\]
which is the number of vertices in the resulting Schreier graph. Thus, for such a $d$ and $N$ sufficiently large, there are 
\[
\geq K\left(\frac{d-v_\Gamma}{e_\Gamma}\right)^{k\frac{d-v_\Gamma}{e_\Gamma}}
\]
conjugacy classes of index $d$ subgroups $H < T$ such that $N_T(H)/H \cong \Gamma$.

In order to conclude for an arbitrary $d$, as in the statement, just consider index $\lfloor d \rfloor_{v_{\Gamma}+e_{\Gamma} N}$ subgroups. Then the above construction of $\cH_\sigma$'s yields the desired result.
\end{proof}

\section{Symmetries of maps, pavings and constellations}\label{section:symmetries}

As mentioned in Section~\ref{intro}, an oriented map $M$ on $n$ darts can be thought of as index $n$ free subgroup $H_M$ of $\Delta^+ = \mathbb{Z}_2*\mathbb{Z}$. Moreover, the group of orientation-preserving automorphisms of $M$ is $\mathrm{Aut}\,M \cong N(H_M)/H_M$, where $N(H_M)$ is the normaliser of $H_M$ in $\Delta^+$, cf.  \cite{BrMeNe} and \cite[Theorem 3.8]{JoSi}. 

An analogous statement holds for several other classes of combinatorial objects, such as 
\begin{itemize}
\item oriented hypermaps, with $\Delta^+ = \mathbb{Z}*\mathbb{Z}$ \cite{MN1, MN2};
\item oriented $(p, q)$-hypermaps, with $\Delta^+ = \mathbb{Z}_p*\mathbb{Z}_q$ ($p\geq 3$, $q\geq 2$) \cite{CK1};
\item oriented pavings (or three-dimensional maps), with $\Delta^+ = \mathbb{Z}_2*\mathbb{Z}_2*\mathbb{Z}_2$ \cite{CK2},
\item length $k\geq 3$ constellations (in the sense of \cite[Definition 1.1.1]{LZ}), with $\Delta^+ = \underbrace{\mathbb{Z}* \dots *\mathbb{Z}}_{k-1}$.
\end{itemize}

The following theorem generalises the respective results of \cite{CM, CM-Survey, Frucht, SirSko} to the case of $(p, q)$-hypermaps.

\begin{thm}\label{thm:p-q-hypermaps-symmetries}
For any finite group $\Gamma$, and $n$ sufficiently large, there exist $\sim n^n$ non-isomorphic oriented $(p,q)$-hypermaps $H$ on $n$ darts with $\mathrm{Aut}\,H \cong \Gamma$. 
\end{thm}

Here by the symbol $f(x) \sim x^x$, for a function $f(x): \mathbb{N} \rightarrow \mathbb{R}$, we mean its \textit{rate of growth}, i.e. that there exist positive constants $A_1, B_1, A_2, B_2$ such that $A_1 x^{B_1 x} \leq f(x) \leq A_2 x^{B_2 x}$, for $x$ sufficiently large. 

The family of $(p,q)$-hypermaps naturally comprises the cases of maps ($p = \infty, q = 2$) and hypermaps ($p = q = \infty$). Some other interesting classes of maps, such as triangulations  ($p=3, q=2$) and quadranqulations ($p=4, q=2$) of surfaces, or their bi-coloured triangulations ($p = q = 3$) also satisfy the above theorem. 

As well, we can now easily estimate the number of three-dimensional pavings, cf. \cite{AK, CK2}, with a given automorphism group.  

\begin{thm}\label{thm:pavings-symmetries}
For any finite group $\Gamma$, and $n$ sufficiently large, there exist $\sim n^n$ non-isomorphic oriented pavings $P$ on $n$ darts with $\mathrm{Aut}\,P \cong \Gamma$. 
\end{thm}

An analogous result holds for constellations as defined in \cite[Definition 1.1.1]{LZ}.

\begin{thm}\label{thm:constellations-symmetries}
For any finite group $\Gamma$, any $k\geq 3$, and $n$ sufficiently large, there exist $\sim n^n$ non-isomorphic length $k$ constellations $C$ on $n$ darts with $\mathrm{Aut}\,C \cong \Gamma$. 
\end{thm}

Since each length $k\geq 3$ constellation defines a branched covering of the sphere $\mathbb{S}^2$ with $k$ branch points, we can reformulate the above theorem in a more geometric language.

\begin{thm}\label{thm:covering-symmetries}
For any finite group $\Gamma$, any $k\geq 3$, and $n$ sufficiently large, there exist $\sim n^n$ non-isomorphic degree $n$ branched coverings of $\mathbb{S}^2$ with $k$ branch points and deck transformation group $\Gamma$. 
\end{thm}

All the above theorems are fairly obvious corollaries of the results in Section~\ref{asymptotic-estimate}. Indeed, the lower bound on the number of non-isomorphic hypermaps, resp. pavings, on $n$ darts with given automorphism group follows from Theorem~\ref{thm:asympt}, and the upper bound of $(n!)^k$, with an appropriate fixed $k\geq 2$, is trivial. Then, invoking Stirling's formula provides the rate of growth $\sim n^n$.

The statements of Theorems~\ref{thm:p-q-hypermaps-symmetries}--\ref{thm:covering-symmetries} should be contrasted with \cite[Lemma 1]{Drmota-Nedela}, which implies that the corresponding combinatorial objects ``mostly'' have only trivial automorphism groups, i.e. are \textit{asymmetric}. However, those with a given non-trivial automorphism group are still numerous, though not as abundant as asymmetric ones. 

Apparently, this technique can be applied to many naturally arising classes of oriented maps, hypermaps, and pavings: exactly those describable as free subgroups of a certain ``universal group'' $\Delta^+$, which is a finite free product of cyclic groups. For any such class we obtain that any finite group is realisable as automorphism group by infinitely many of its members (more precisely, super-exponentially many depending on the number of darts). It is worth mentioning that not all families of maps admit such a wide variety of symmetries, e.g. the maps with underlying graph a tree \cite{F}. 

\longthanks{
The authors would like to thank Derek Holt (University of Warwick, UK) for answering an author's question on math.stackexchange \footnote{https://math.stackexchange.com/questions/2767720/preimage-of-product-of-free-subgroups-in-free-product-is-free}, Laura Ciobanu (Heriot-Watt University, UK), Adrian Tanasa (Universit\'e Bordeaux, France) and Gareth A. Jones (University of Southampton, UK) for numerous fruitful discussions, and the anonymous referee for helpful remarks and suggestions.
}

\nocite{*}
\bibliographystyle{amsplain-ac}
\bibliography{biblio}
\end{document}